\documentclass{article}

\usepackage{amssymb}
\usepackage{amsmath}
\usepackage{amsthm, amsfonts, mathrsfs}
\usepackage{xcolor}

\usepackage{enumitem}

\usepackage{float}

\usepackage{cmll}

\usepackage{tikz}

\usetikzlibrary{calc}

\newtheorem{thm}{Theorem}[section]
\newtheorem{cor}[thm]{Corollary}

\newtheorem{prop}[thm]{Proposition}

\newtheorem{lem}[thm]{Lemma}

\newtheorem{conj}[thm]{Conjecture}

\newtheorem{qu}{Question}

\newcommand{\comment}[1]{}

\newcommand{\lb}{\left\{}
\newcommand{\rb}{\right\}}
\newcommand{\la}{\left<}
\newcommand{\ra}{\right>}

\newcommand{\tq}{\ge_T}
\newcommand{\te}{=_T}

\newcommand{\gtq}{\ge_{T}}
\newcommand{\gte}{=_{T}}

\newcommand{\K}{\mathcal{K}}

\newcommand{\omom}{\omega^\omega}

\newcommand{\cl}[1]{\overline{#1}}

\DeclareMathOperator{\nci}{\neg c}

\DeclareMathOperator{\ii}{ii}

\DeclareMathOperator{\cof}{cof}

\title{The Cofinality of Generating Families}
\author{Paul Gartside and Thomas Gilton}

\begin{document}

\maketitle
\begin{abstract}
The topology of an infinite, separable metrizable space $M$ is \emph{generated} by a family $\mathcal{C}$ of its subsets provided that a set $A\subseteq M$ is closed in $M$ if and only if $A\cap C$ is closed in $C$ for each $C\in \mathcal{C}$. The \emph{sequentiality number}, $\mathop{seq}(M)$, and \emph{$k$-ness number}, $\mathop{k}(M)$,   of $M$, are the minimum size of a generating family of convergent sequences, respectively compact subsets. 

Let $\mathfrak{b}$ be the minimum size of an unbounded set in $\omega^\omega$ with the mod finite order. For a cardinal $\kappa$, the \emph{covering number}, $\mathrm{cov}(\kappa)$, is the minimum size of a family of countable subsets of $\kappa$ so that every countable subset of $\kappa$ is contained in an element of the family. It is shown using the Tukey order on relations that (1) $\mathrm{seq}(M)=\mathrm{cov}(|M|)\cdot \mathfrak{b}$,  unless $M$ is locally small (every point of $M$ has a neighborhood of size strictly less than $|M|$) in which case $\mathrm{seq}(M)=\lim_{\mu <|M|} \mathrm{cov}(\mu)\cdot \mathfrak{b}$ and 
(2) $k(M)$ is in the interval $[kc(M)\cdot\mathfrak{b},\mathrm{cov}(kc(M))\cdot \mathfrak{b}]$, where $kc(M)$ is the minimum number of compact sets that cover $M$. 

Solutions to problems of van Douwen's on the $k$-ness number of analytic and of co-analytic spaces are deduced. 


\end{abstract}

\section{Introduction}
For any subset $M$ of the real line its convergent sequences determine which subsets of $M$ are closed,   hence the convergent sequences determine the topology, and hence the closure of any set and which real valued functions on $M$ are continuous. It is natural to ask \emph{how many sequences suffice to generate the topology?}  
More concretely, the topology of a space $M$ (all spaces in this paper are separable and metrizable)   is \emph{generated} by a family $\mathcal{C}$ of its subsets provided that a subset $A$ of $M$ is closed in $M$ if and only if $A\cap C$ is closed in $C$ for each $C\in \mathcal{C}$. Then $\mathrm{seq}(M)$, the \emph{sequentiality number} of $M$, is the minimum size of a family of convergent sequences in $M$ (specifically, subsets homeomorphic to the subspace $\{0\}\cup\{1/n : n \in \mathbb{N}\}$ of the reals)  generating the topology of $M$. 
It is also natural to define  $\mathop{k}(M)$, the \emph{$k$-ness number} of $M$, as the minimum size of a generating family  of compact sets. In this paper we compute $\mathrm{seq}(M)$ and $k(M)$ in terms of intrinsic properties of the separable metrizable space $M$.

Van Douwen introduced the $k$-ness number  
in his survey article, \cite{vanDou}, on cardinal characteristics of the continuum and small cardinals arising in topology. Observe that, for any separable metrizable space  $M$, we have $k(M) \le \mathrm{seq}(M) \le \mathfrak{c}$, the continuum ($\mathfrak{c}=|\mathbb{R}|=2^{\aleph_0}$). Van Douwen noted (in \cite{vanDou}) that $k(M) \le \omega$ only when $M$ is locally compact, and he showed that when  $M$ is not locally compact we have $\mathfrak{b} \le k(M)$, where $\mathfrak{b}$ is the \emph{bounding number} (minimum size of an unbounded set in $\omega^\omega$ with the mod finite order, $<^*$). 
Van Douwen computed the $k$-ness number for a handful of spaces, notably the rationals, but left open (see \cite[Question~8.11]{vanDou}) the computation of $k(M)$ for $M$ analytic or co-analytic. We conclude the paper with the solution to van Douwen's question. 

Our results depend on two sources of techniques: first the theory of relations, specifically the Tukey order,   and second Shelah's \textsc{PCF}  theory. 
We now give an overview of how these techniques combine to enable us to compute the sequentiality and $k$-ness numbers. It is convenient to do so initially for $\mathrm{seq}(M)$, where our results are essentially complete (up to the limits of our understanding of \textsc{PCF} theory); and then for $k(M)$, where our results loosely mirror those for sequentiality.

From work by the first author and Ziqin Feng in \cite{FG_Shape} we have a detailed understanding, mediated by the Tukey order, of the relations $\mathbf{seq}(M)$ and $\mathbf{k}(M)$,  whose cofinal sets are generating families of convergent sequences and compact subsets, respectively, and hence whose cofinalities (minimum size of a cofinal set) are $\mathrm{seq}(M)$ and $k(M)$, respectively. 
 A map $\phi_+$ from one relation, $\mathbf{A}$, to another, $\mathbf{B}$, is a Tukey morphism if it carries cofinal sets to cofinal sets. The Tukey order is given by $\mathbf{A} \tq \mathbf{B}$ if there is a Tukey morphism of the first relation to the second. If $\mathbf{A} \tq \mathbf{B}$ then the cofinality of $\mathbf{A}$ is at least as big as the cofinality of $\mathbf{B}$. Consequently the Tukey order is a convenient tool to compare and compute cofinalities of relations such as $\mathbf{seq}(M)$ and $\mathbf{k}(M)$. 
(See Section~\ref{ssec:prelim_rel} for background material on relations and Tukey morphisms.) 

The outcome for the sequentiality number of a separable metrizable space $M$ is that $\mathrm{seq}(M)=\mathrm{sam}(|M|)\cdot \mathfrak{b}$,  unless $M$ is locally small (every point of $M$ has a neighborhood of size strictly less than $|M|$) in which case $\mathrm{seq}(M)=\lim_{\mu <|M|} \mathrm{sam}(\mu)\cdot \mathfrak{b}$. 
Here, for any cardinal $\kappa$, the \emph{sampling number} of $\kappa$ is the least number of countable subsets of $\kappa$ needed to have infinite intersection with every countably-infinite subset of $\kappa$.  Evidently the sampling number is a purely set-theoretic object. It is easily seen that $\mathrm{seq}$ is ranked by the size of the space: if $|M| < |N|$ then $\mathrm{seq}(M) \le \mathrm{seq}(N)$. Also among all spaces $M$ of a fixed cardinality, $\mathrm{seq}(M)$ can only take on at most two values, and can attain more than one only when $|M|$ has countable cofinality. 

To compute the sequentiality number, then, we need to understand the sampling number and its relation to the bounding number, $\mathfrak{b}$, and the continuum, $\mathfrak{c}$ (because separable metrizable spaces have cardinality no more than $\mathfrak{c}$).  
One of the fundamental objects of \textsc{PCF} theory is the \emph{covering number} of a cardinal $\kappa$ which is the minimal size of a family  of countable subsets of $\kappa$ such that every countable subset is contained in a member of the family. 
Clearly the sampling and covering numbers are closely related. In fact, see Theorem~\ref{th:cov=sam}, despite the formal difference between them, they coincide: for every $\kappa$ we have $\mathrm{sam}(\kappa)=\mathrm{cov}(\kappa)$. 
From \textsc{PCF} theory much is known about the covering number and we recapitulate this next. However 
our results above on the sequentiality number, as well as those below on the $k$-ness number, are most naturally established for the sampling number. 
Since the sampling and covering numbers are equal, we can substitute one for the other at our convenience. 
For consistency we have stated our results below using the covering number, but in proofs often use the sampling number instead without explicitly (and repeatedly) stating the equality.  The concluding paragraph of Section~\ref{ssec:seqsam} gives details on how our understanding of the covering number informs our knowledge of the sequentiality number.

It is known that $\kappa \le \mathrm{cov}(\kappa)$, but $\mathrm{cov}(\kappa) \ge \kappa^+$ for uncountable $\kappa$ with countable cofinality. Further, it is a basic tenet of \textsc{PCF} theory that the covering number is `tame' - $\mathrm{cov}(\kappa)$ does not jump up much above $\kappa^+$. 
Indeed, at least assuming the Continuum Hypothesis, \textsc{CH},  large cardinal strength is required  to get $\mathrm{cov}(\kappa) > \kappa^+$. However, it is also known that, assuming the existence of a supercompact cardinal, it is consistent that \textsc{CH} holds and  $\mathrm{cov}(\kappa) > \kappa^+$ for $\kappa=\aleph_\omega$. 
Applying random real forcing to the latter model,  allows us to construct, again modulo large cardinals, models where, for example,  $\mathfrak{b}=\aleph_1$ (and so is a non-factor in the possible values of $\mathrm{seq}$), $\mathrm{cov}(\aleph_\omega)>\aleph_{\omega+1}$ and $\mathfrak{c}$ is as large as we like. Section~\ref{ssec:pcf} provides  details on the above, and concludes with a discussion of the cofinality of products of countable sets of cardinals with the mod finite order, and in particular the existence of mod finite scales. 
This enables the construction of models, modulo large cardinals, in which we can control the covering numbers of families of cardinals.

Finally, in Section~\ref{ssec:ksam} we turn to the question of how many compact subsets suffice to generate the topology of a separable metrizable space. 
Let $kc(M)$, the \emph{compact covering number}, be the minimum number of compact subsets needed to cover a separable metrizable $M$. 
The case of $\sigma$-compact spaces (i.e. $kc(M) \le \omega$) was known to van Douwen, so we move on to  $M$ with $kc(M)$ uncountable.  
For such $M$ it is no surprise that $kc(M) \le k(M)$, but it was a surprise to the authors that $k(M) \le \mathrm{cov}(kc(M))\cdot \mathfrak{b}$, and so - as discussed above in  the context of the sequentiality number - the compact covering and $k$-ness numbers are tightly coupled. But less so than the sequentiality number of a space and its cardinality. 
Recall that for a fixed $\kappa$, $\mathrm{seq}(M)$, for  any $M$ with $|M|=\kappa$, is constrained to at most two values. However we show, modulo large cardinals, for every finite $K\ge 2$  it is consistent that,  for a fixed $\kappa$, $k(M)$, for  any $M$ with $kc(M)=\kappa$, takes on exactly $K$ values, and it is consistent that,  for a fixed $\kappa$, there are infinitely many values of $k(M)$, for   $M$ with $kc(M)=\kappa$. 
Section~\ref{ssec:ksam} concludes with the solution to van Douwen's questions on the $k$-ness numbers of analytic and co-analytic spaces.

\section{Preliminaries}
\subsection{Relations and Tukey Order}\label{ssec:prelim_rel} 
\paragraph{Relations, Cofinality, Morphisms}
A \emph{relation} is a triple $\mathbf{A}=(A_-,A_+,A)$, where $A_-$ is the domain, $A_+$ is the co-domain, and $A$ is a collection of pairs  (a \emph{general relation}). 
Write $xAy$ if $(x,y) \in A$. 
Abbreviate to $(A',A)$  a relation $(A_-,A_+, A)$  where $A_-=A'=A_+$. 
A subset $C$ of $A_+$ is \emph{cofinal} in $\mathbf{A}$ if for every $x$ from $A_-$ there is a $y$ from $C$ such that $xAy$, and  the \emph{cofinality} of $\mathbf{A}$, denoted $\cof(\mathbf{A})$, is the minimum size of a cofinal set. 
A subset $S$ of $A_-$ is \emph{bounded} if for some $y$ in $A_+$, for all $x$ from $S$ we have $xAy$. 
The \emph{additivity} of a relation $\mathbf{A}$, denoted $\mathrm{add}(\mathbf{A})$, is the minimum size of an unbounded set.  
Equivalently, it is the cofinality of the dual $\mathbf{A}^\perp=(A_+,A_-,\neg A^{-1})$ of $\mathbf{A}$.

A \emph{Tukey morphism} from $\mathbf{A}$ to $\mathbf{B}=(B_-,B_+,B)$ 
was  originally  defined (see \cite{Vojtas}), to be a pair $( \psi_-,\phi_+)$ of functions such that (i) $\phi_+: A_+ \to B_+$, (ii) $ \psi_-: B_- \to A_-$ and (iii) $\psi_-(b) A a \implies b B \phi_+(a)$, where  $b \in B_- \ \text{and}\ a \in A_+$. 
Note that the map $\phi_+$ of a Tukey morphism pair 
carries cofinal sets to cofinal sets.  
It can be shown conversely that if $\phi_+:A_+ \to B_+$  carries cofinal sets to cofinal sets then there exists a $\psi_-$ so that $(\psi_-,\phi_+)$ is a Tukey morphism. Consequently we also call a function $\phi_+:A_+ \to B_+$  mapping cofinal sets to cofinal sets a Tukey morphism. 
If such a morphism exists we write $\mathbf{A} \gtq \mathbf{B}$.  
The relation $\gtq$, the \emph{Tukey order} on relations,  is transitive, hence $\gte$, defined by $\mathbf{A} \gte \mathbf{B}$ if and only if $\mathbf{A} \gtq \mathbf{B}$ and $\mathbf{B} \gtq \mathbf{A}$, is an equivalence relation,  \emph{Tukey equivalence} on relations, whose equivalence classes are called \emph{Tukey types}.

Suppose $(\psi_-,\phi_+)$ is a Tukey morphism pair from $\mathbf{A}$ to $\mathbf{B}$. 
Then $\phi_+$ maps cofinal sets to cofinal sets, and so $\cof(\mathbf{A}) \ge \cof(\mathbf{B})$. 
One can check that $(\phi_+,\psi_-)$ is a Tukey morphism pair from $\mathbf{B}^\perp$ to $\mathbf{A}^\perp$, and so $\mathrm{add}(\mathbf{A}) \le \mathrm{add}(\mathbf{B})$.

For any set $S$, equality on $S$ gives a relation, $(S,S,=)=(S,=)$. Note $(S,=) \te (S',=)$ if and only if $|S|=|S'|$. 
If $\mathcal{S}$ is a family of sets then $(S,\mathcal{S},\in)$ is a relation whose cofinal sets are the covers of $S$ from $\mathcal{S}$. 
On $\omega^\omega$ let $<^*$ be the mod finite order, so $f<^*g$ if $f(n)<g(n)$ except for finitely many $n$. Then the \emph{dominating number}, $\mathfrak{d}$, is $\cof(\omega^\omega,<^*)$; while the \emph{bounding number}, $\mathfrak{b}$, is the additivity of $(\omega^\omega,<^*)$. 

\paragraph{Operations on Relations}  
Given relations $\mathbf{A}$ and $\mathbf{B}$ define 
$\mathbf{A} \times \mathbf{B}$ to be the standard product, $(A_-\times B_-,A_+\times B_+,A\times B)$, and 
$\mathbf{A} + \mathbf{B}$,  to be $(A_- \oplus B_-,A_+ \oplus B_+,S)$ where $x S y$ if and only if either $x \in A_-$, $y \in A_+$ and $xAy$ or $x\in B_-$, $y\in B_+$ and $xBy$. 
More generally, if $\{\mathbf{A}_\lambda=((A_\lambda)_-,(A_\lambda)_+,A_\lambda) : \lambda \in \Lambda\}$ is a family of relations then its sum is $\sum_{\lambda \in \Lambda} \mathbf{A}_\lambda = (\bigoplus_{\lambda \in \Lambda} (A_\lambda)_-,\bigoplus_{\lambda \in \Lambda} (A_\lambda)_+, S)$ where $x S y$ if and only if for some $\lambda$, $x\in (A_\lambda)_-, y \in (A_\lambda)_+$ and $xA_\lambda y$.
We need the following facts about the sum operation:
\begin{enumerate}[noitemsep,topsep=1pt]
    \item[(A)] If for all $\lambda\in\Lambda$, $\mathbf{A}_\lambda = \mathbf{A}$, then 
    $
    \sum_{\lambda \in \Lambda} \mathbf{A}_\lambda \gte \mathbf{A} \times (|\Lambda|,=)$. 
    \item[(B)]  If  $\Lambda'$ is a subset of $\Lambda$ then 
    $
    \sum_{\lambda \in \Lambda} \mathbf{A}_\lambda \gtq \sum_{\lambda \in \Lambda'} \mathbf{A}_\lambda$. 
    \item[(C)]  If for every $\lambda$ in $\Lambda$,  $\mathbf{A}_\lambda \gtq \mathbf{B}_\lambda$, then 
    $
    \sum_{\lambda \in \Lambda} \mathbf{A}_\lambda \gtq \sum_{\lambda \in \Lambda} \mathbf{B}_\lambda$. 
    \item[(D)] The cofinality of $\sum_{\lambda\in\Lambda}\mathbf{A}_\lambda$ is equal to the sum of the cofinalities of the $\mathbf{A}_\lambda$.
\end{enumerate}

\paragraph{Some Special Relations} 
Let $(S_n)_n$ be a sequence of sets. Let $\left(\prod_n [S_n]^{<\omega}\right)_\infty$ be all $(F_n)_n$ in the product such that $F_n\ne \emptyset$ for infinitely many $n$.
Define a general relation $i_\infty$ on sequences of sets by $(F_n)_n i_\infty (G_n)_n$ if and only if for infinitely many $n$ we have $F_n \cap G_n \ne \emptyset$. 
When all the $S_n$ are equal, say to $S$, then we write $\left([S]^{<\omega}\right)^{\omega}_\infty$ for $\left(\prod_n [S_n]^{<\omega}\right)_\infty$.

For a cardinal $\kappa$, we now define the three relations $P(\kappa)$, $Q(\kappa)$, and $S(\kappa)$, the latter two of which have subcases based on whether or not $\kappa$ has countable cofinality.
\begin{enumerate}
    \item Let $P(\kappa)=(\left([\kappa]^{<\omega}\right)^{\omega}_\infty,i_\infty)$.
    \item If $\cof(\kappa) \ne \omega$ then set $Q(\kappa)$ and $S(\kappa)$ to be the trivial relation $(\emptyset,\emptyset,\emptyset)$.
    \item Otherwise, when $\cof(\kappa)=\omega$, fix a strictly  increasing sequence of infinite cardinals, $(\kappa_n)_n$ with limit $\kappa$. We now set
    $$
    Q(\kappa)=\left(\left(\prod_n [\kappa_n]^{<\omega}\right)_\infty,i_\infty\right)
    $$
    and
    $$
    S(\kappa)=\sum_n P(\kappa_n) \times (\kappa_n,=)
    $$
\end{enumerate}
From \cite[Lemma~2.14]{FG_Shape} we know $Q(\kappa)$ and $S(\kappa)$ are well-defined in the case when $\kappa$ has countable cofinality, because their  Tukey type is independent of the choice of cofinal sequence, $(\kappa_n)_n$.
We know \cite[Lemma~2.10]{FG_Shape} that $P(\omega) \te (\omom,\le_\infty)$, where 
$$
f \le_\infty g\;\text{ if }\; \{n \in \omega : f(n) \le g(n)\}\text{ is infinite,}
$$
and hence $\cof(P(\omega))=\mathfrak{b}$. We end this subsection with the following simple lemma:

\begin{lem}\label{lem:cofatleastkappa}
    $\cof(P(\kappa))$ is at least $\kappa$. 
\end{lem}
    \begin{proof} Suppose that $\mathcal{C}' \subseteq ([\kappa]^{<\omega})^{\omega}_{\infty}$ has size strictly less than $\kappa$. Then the set $S=\bigcup \{ \bigcup_n F_n : (F_n)_n \in \mathcal{C}'\}$ has size strictly less than $\kappa$, so pick $\alpha \in \kappa \setminus S$. 
Now we see that $(\{\alpha\})_n$ witnesses that $\mathcal{C}'$ is not cofinal.
\end{proof}

\subsection{\textsc{PCF} Theory}\label{ssec:pcf} 

Two elements of  Shelah's \textsc{PCF} theory 
play a key role in our results, namely the covering number and mod finite scales.

\paragraph{The Covering Number} 
Relations of the form $([\mu]^{\delta},\subseteq)$, and their cofinalities, where $\mu$ is singular and  $\delta\leq\mu$, are central objects of \textsc{PCF} theory. For the purposes of this paper, the case $\delta=\omega$ is the most important, and, simplifying terminology and notation, we define the \emph{covering number}, $\mathrm{cov}(\kappa)$, of a cardinal $\kappa$, to be $\cof([\kappa]^\omega,\subseteq)$. Basic facts about the covering number are recorded below, and follow quickly from the definitions.

\begin{lem}\label{lem:samIncreasing}
    If $\kappa\geq\lambda$, then $\mathrm{cov}(\kappa)\geq\mathrm{cov}(\lambda)$.
\end{lem}

\begin{lem} \label{l:sam1} 
Let $\kappa$ be an uncountable cardinal. Then
\begin{enumerate}[noitemsep,topsep=1pt]
\item[(i)]  $\kappa \le \mathrm{cov}(\kappa) \le \kappa^\omega$, and $\mathrm{cov}(\kappa^\omega)=\kappa^\omega$, 

\item[(ii)] if $\cof(\kappa)=\omega$ then $\mathrm{cov}(\kappa) > \kappa$,

\item[(iii)] if $\kappa=\mu^+$ then     $\mathrm{cov}(\kappa) = \mathrm{cov}(\mu)\cdot\kappa$,  and

\item[(iv)] if $\kappa$ is a limit and $\cof(\kappa)>\omega$  then $\mathrm{cov}(\kappa)=\lim \{\mathrm{cov}(\mu) : \mu<\kappa\}$. 
\end{enumerate}
\end{lem}

\begin{lem}\label{l:sam_aleph_n}
    For every $n \in \mathbb{N}$ we have $\mathrm{cov}(\aleph_n)=\aleph_n$.
\end{lem}

In several cases, \textsc{PCF} theory places strong upper bounds on the possible values of covering numbers.
The following proposition, which is surely folklore, gives an example.

\begin{prop}\label{pr:fixpt}
    Let $\lambda$ be the least fixed point of the $\aleph$-function. 
    Let $\kappa<\lambda$ be infinite, and write $\kappa=\aleph_\alpha$. 
    Then $\mathrm{cov}(\kappa)=\cof([\kappa]^\omega,\subseteq)<\aleph_{|\alpha|^{+4}}$.
\end{prop}
\begin{proof}
    We will need several facts, mostly from \textsc{PCF} theory, which we list here: we include and which can be found in a standard \textsc{PCF} reference, such as \cite{AbrahamMagidor}:
    \begin{enumerate}[nosep,topsep=1pt]
        \item[(1)] 
        $\cof([\mu]^{\nu_1},\subseteq)\leq\cof([\mu]^{\nu_2},\subseteq)\cdot\cof([\nu_2]^{\nu_1},\subseteq)$, when
        $\mu$ is infinite and $\nu_1\leq\nu_2$ (\cite{AbrahamMagidor}, Lemma 5.10).
        \item[(2)] Suppose that $\mu$ is a singular cardinal and $\kappa<\mu$ is any infinite cardinal so that the interval $A:=(\kappa,\mu)\cap\textsc{REG}$ of regular cardinals has size $\leq\kappa$.
        Then 
        $\cof([\mu]^{\kappa},\subseteq)=\operatorname{maxpcf}(A)$ (\cite{AbrahamMagidor}, Theorem 5.11).
        \item[(3)] Suppose that $A$ is a progressive interval of regular cardinals (\cite{AbrahamMagidor}, Theorem 3.9).\footnote{This means that (1) $A$ has the form $[\kappa,\lambda)\cap\text{REG}$, where REG denotes the class of regular cardinals and where $\kappa<\lambda$ are cardinals, and (2) $|A|<\min(A)$.}
        Then $\operatorname{pcf}(A)$ is an interval of regular cardinals.
        \item[(4)] If $A$ is a progressive interval of regular cardinals, then 
        $|\operatorname{pcf}(A)|<|A|^{+4}$ (\cite{AbrahamMagidor}, Theorem 7.1).        
    \end{enumerate}

We prove the proposition by induction on $\kappa\geq\omega$. The result is trivial for $\kappa=\omega$, since the relevant cofinality is simply equal to 1.
We next take care of the successor case. 
Suppose that $\kappa=\aleph_{\alpha+1}$, for some $\alpha<\lambda$ and that the result is known for $\aleph_\alpha$.
Then
$
\cof([\aleph_{\alpha+1}]^\omega,\subseteq)=\cof([\aleph_\alpha]^\omega,\subseteq)\cdot\aleph_{\alpha+1}
$.
By induction, $\cof([\aleph_\alpha]^\omega,\subseteq)<\aleph_{|\alpha|^{+4}}$.
We also trivially have $\alpha+1<|\alpha|^{+4}$.
Thus both terms on the right-hand side are below $\aleph_{|\alpha|^{+4}}\leq\aleph_{|\alpha+1|^{+4}}$, which completes the proof in the successor case.

We now address the limit case.
Suppose that $\kappa=\aleph_\alpha$ where $\alpha$ is a limit.
To structure the following argument, we begin with an application of fact (1) above, namely that 
$(\ast)\ \cof([\kappa]^\omega,\subseteq)\leq\cof([\kappa]^{|\alpha|},\subseteq)\cdot\cof([|\alpha|]^\omega,\subseteq)
$.
Our argument will proceed by bounding both of the terms on the right-hand side of $(\ast)$.

To begin, note that $\alpha<\aleph_\alpha$, since $\kappa<\lambda$ and $\lambda$ is the least fixed point of the $\aleph$-function.
Since $\alpha$ is a limit, $|\alpha|$ is infinite.
Also, the interval $A:=(|\alpha|,\kappa)\cap\textsc{REG}$ of regular cardinals has size $|\alpha|$: indeed, as $\kappa=\aleph_\alpha$, the function taking each $\nu\in A$ to its index $\beta$ in the $\aleph$-sequence is an injection of $A$ into $\alpha$.
By fact (2) above, we conclude that $\cof([\kappa]^{|\alpha|},\subseteq)=\operatorname{maxpcf}(A)$.
As $A$ is a progressive (since $|A|=|\alpha|<\min(A)$) interval of regular cardinals, we conclude from facts (3) and (4) that $\operatorname{pcf}(A)$ is an interval of regular cardinals of size $<|A|^{+4}=|\alpha|^{+4}$.

We now claim that $\operatorname{maxpcf}(A)<\aleph_{|\alpha|^{+4}}$.
Suppose otherwise, for the sake of a contradiction.
Certainly $\operatorname{pcf}(A)$ contains $\min(A)$ and $\operatorname{maxpcf}(A)$.
Since $\operatorname{pcf}(A)$ is an interval of regular cardinals, it therefore contains the interval $[\min(A),\operatorname{maxpcf}(A)]\cap\text{REG}$.
Since $\aleph_{|\alpha|^{+4}}<\operatorname{maxpcf}(A)$, $\operatorname{pcf}(A)$ therefore contains the interval $[\min(A),\aleph_{|\alpha|^{+4}})\cap\text{REG}$ of regular cardinals.
As $\min(A)=|\alpha|^+<\aleph_{|\alpha|^{+4}}$, we conclude that $\operatorname{pcf}(A)$ has at least $|\alpha|^{+4}$-many elements, contradicting the fact that $|\operatorname{pcf}(A)|<|\alpha|^{+4}$.

Bound the first term on the right-hand side of $(\ast)$, by stringing together the previous results to see 
$\cof([\kappa]^{|\alpha|},\subseteq)=\operatorname{maxpcf}(A)<\aleph_{|\alpha|^{+4}}$.

To bound the second term, we use the induction hypothesis.
Write $|\alpha|=\aleph_\beta$ for some $\beta<|\alpha|$. 
By induction $\cof([|\alpha|]^\omega,\subseteq)<\aleph_{|\beta|^{+4}}$.
Since $\beta<|\alpha|$ (because $|\alpha|$ is not a fixed point of the $\aleph$-function), we know that $|\beta|^{+4}<|\alpha|^{+4}$.
Thus $\cof([|\alpha|]^\omega,\subseteq)<\aleph_{|\alpha|^{+4}}$.

We  conclude that both terms on the right-hand side of $(\ast)$ are $<\aleph_{|\alpha|^{+4}}$, and this completes the induction.
\end{proof}

Here we note that the situation is more complex at fixed points of the $\aleph$-function. Assuming the consistency of certain large cardinal assumptions, $\cof([\kappa]^\omega,\subseteq)$ can be arbitrarily large when $\kappa$ is the first fixed point of the $\aleph$-function; this result follows almost immediately from the following theorem, due to Gitik:

\begin{thm} (Gitik, \cite{Gitik:noBounds}) Let $\kappa$ be a cardinal of cofinality $\omega$ with certain large cardinals below $\kappa$.\footnote{Namely, for each $\tau<\kappa$, the set $\lb\alpha<\kappa:o(\alpha)\geq\alpha^{+\tau}\rb$ is unbounded in $\kappa$; note that $o(\alpha)$ is the Mitchell order of a cardinal $\alpha$.} Let $\lambda>\kappa$ be any cardinal. Then there is a forcing extension in which $\kappa$ is the first fixed point of the $\aleph$-function, $2^\kappa\geq\lambda$, and the \textsc{GCH} holds below $\kappa$.
\end{thm}

\begin{cor} (Gitik, \cite{Gitik:noBounds})\label{cor:noBounds}
    Assuming the existence of certain large cardinals, $\mathrm{cov}(\kappa)$ can be arbitrarily large, where $\kappa$ is the first fixed point of the $\aleph$-function, and \textsc{CH} holds.
\end{cor}
\begin{proof}
    Work in Gitik's model from above where $\kappa$ is the first fixed point of the $\aleph$-function, \textsc{GCH} holds below $\kappa$, and $2^\kappa\geq\lambda$ (where $\lambda>\kappa$ is arbitrary). 
    We first note that $\kappa^\omega=\cof([\kappa]^\omega,\subseteq)\cdot2^{\aleph_0}=\cof([\kappa]^\omega,\subseteq)$, using the fact that $2^{\aleph_0}=\aleph_1$ in the model under consideration. Next observe that $\kappa^\omega=2^\kappa$, using the \textsc{GCH} below $\kappa$. Putting these together, we conclude that $\cof([\kappa]^\omega,\subseteq)=2^\kappa\geq\lambda$.
\end{proof}

Optimal upper bounds follow from `no large cardinal' hypotheses. 
Indeed, by work of Gitik (see the ``Main Theorem" in \cite{Gitik:Strength}), we know that if there does not exist a cardinal $\kappa$ with Mitchell order $\kappa^{++}$ (written $o(\kappa)=\kappa^{++})$, and if $\mu$ is a singular cardinal of cofinality $\omega$, \emph{and} if $\mathfrak{c}=2^\omega<\mu$ (note the cardinal arithmetic assumption here), then $\mu^\omega=\mu^+$.  
For such a $\mu$, we conclude that $\mu^+\leq\mathrm{cov}(\mu)\leq\mu^\omega=\mu^+$, giving us the value of $\mu^+$ for $\mathrm{cov}(\mu)$.
By induction, we obtain the following theorem:

\begin{thm}[Gitik]\label{thm:boringcov}
    If there is no $\lambda$ with Mitchell order $\lambda^{++}$, and if $\mathfrak{c}<\aleph_\omega$, then for every uncountable cardinal $\kappa$,
    \[ \mathrm{cov}(\kappa) = 
    \begin{cases}
        \kappa^+ & \text{ if }\cof(\kappa)=\omega \\
        \kappa & \text{ otherwise.}
    \end{cases}\]
\end{thm}

In the case that there is no $\kappa$ with $o(\kappa)=\kappa^{++}$, but $\mathfrak{c}>\aleph_\omega$, we have the same behavior for the $\mathrm{cov}$ function starting at least as low as the least singular above $2^\omega$.

In the context of computing the sequentiality and $k$-ness numbers, our interest is in $\mathrm{cov}(\kappa)$ where $\kappa$ (the size of a separable metrizable space) is no more than $\mathfrak{c}$. 
Unfortunately, then, the preceding theorem does not apply. 
\begin{qu}\label{qu:large_card}\hfill
\begin{enumerate}[noitemsep,topsep=1pt]
\item If there is no $\lambda$ with Mitchell order $\lambda^{++}$, then is $\mathrm{cov}(\aleph_\omega)=\aleph_{\omega+1}$? 
\item Next we have a version of the previous item, but for other cardinals than $\aleph_\omega$ specifically: suppose that there is no $\lambda$ with Mitchell order $\lambda^{++}$, and that $\mathfrak{c}>>\aleph_{\omega}$. Is it true that for every $\omega<\kappa<\mathfrak{c}$ of countable cofinality, $\mathrm{cov}(\kappa)=\kappa^+$? 
\end{enumerate}
\end{qu}

In personal communication, told to us via Cummings, Gitik thinks (\cite{Gitik:personal}) that the answer to Question A, part 2, is positive, i.e., that the \textsc{CH} is not necessary.

Assuming the existence of sufficiently large cardinals, we can construct models in which the covering number exceeds this optimal upper bound, for example at $\aleph_\omega$.
Indeed, Magidor \cite{Magidor} showed that if there is a supercompact cardinal then it is consistent that \textsc{CH} holds and $\mathrm{cov}(\aleph_\omega)>\aleph_{\omega}^+$. 
Gitik 
\cite{Gitik:NegSCH} was able to lower the large cardinal strength, which combined with his result above 
gives the exact consistency strength of the failure of $\mathsf{SCH}$.

\paragraph{Controlling $\mathfrak{b}$ and $\mathfrak{c}$}
The models above where $\mathrm{cov}(\kappa) > \kappa^+$ is not directly relevant for our purposes (computing $\mathrm{seq}(M)$ and $k(M)$, for $M$ separable metrizable) because \textsc{CH} holds, whereas we need the continuum large while also controlling $\mathfrak{b}$. 
To correct this we modify  via forcing, intending to keep $\mathfrak{b}$ and the relevant covering numbers unchanged, while boosting the continuum.

As is well known c.c.c. forcings do not change the covering number. 
\begin{lem}\label{lem:cccPreservesCapturing}
    Let $\kappa$ be an uncountable cardinal.
        Let $\mu=\mathrm{cov}(\kappa)$ and let $\mathbb{P}$ be c.c.c. Let $\dot{X}$ be the $\mathbb{P}$-name for the set of countable subsets of $\kappa$ in the extension. Then $\mathbb{P}$ forces that $\mu=\cof(\dot{X},\subseteq)$.
\end{lem}

In particular, we can use random real forcing to push the continuum high, while keeping the bounding number low, and fixing the covering number. 
\begin{thm}\label{th:small_b_big_sam}
    If there is a supercompact cardinal then it is consistent that $\mathfrak{b}=\mathfrak{d}=\aleph_1$, $\mathfrak{c}$ as large as we like, and $\mathrm{cov}(\aleph_\omega)>\aleph_{\omega+1}$.
\end{thm}
\begin{proof}
    Take Magidor's model from \cite{Magidor} and recall in this model the \textsc{CH} holds, and $\mathrm{cov}(\aleph_\omega)>\aleph_{\omega+1}$.
        Recalling that random real forcing over a model of the \textsc{CH} preserves that $\mathfrak{b}=\mathfrak{d}=\omega_1$ (\cite{Kunen2011}, Lemma IV.7.36), we force over Magidor's model to keep $\mathfrak{b}$ and $\mathfrak{d}$ equal to $\omega_1$, push $\mathfrak{c}$ as large as we like, and not change $\mathrm{cov}(\aleph_\omega)>\aleph_{\omega+1}$. This last claim follows from Lemma~\ref{lem:cccPreservesCapturing}.
\end{proof}

A similar argument gives the following.
\begin{thm}\label{th:sam_of_b}
    If there is a supercompact cardinal then it is consistent that $\mathfrak{b}=\mathfrak{d}=\aleph_{\omega+1}$, $\mathfrak{c}$ as large as we like, and $\mathrm{cov}(\aleph_{\omega+1})>\mathfrak{b}$.
\end{thm}
\begin{proof}
    Start with Magidor's model from \cite{Magidor} as in the previous Theorem. Do a finite-support iteration of Hechler forcing (\cite{Hechler1974}) of length $\aleph_{\omega+1}$, noting that since this is a finite-support iteration of c.c.c. posets, it is c.c.c. as well. This iteration makes $\mathfrak{b}=\mathfrak{d}=\aleph_{\omega+1}$ (see \cite{Blass2010}, especially Theorem 2.6, and Table~4, p.~480).\footnote{For the fastidious reader: Blass discusses this iteration of Hechler forcing over a model of the $\mathrm{GCH}$, in which the iteration has length $\lambda$, a regular cardinal at least $\omega_2$. The GCH is used to give a clean argument for the value of $\mathfrak{c}$. However, since this iteration is itself c.c.c., we may still iterate with length $\lambda$ regular, preserving all cardinals, to make $\mathfrak{b}=\mathfrak{d}=\lambda$. If in addition, $\lambda^{\aleph_0}=\lambda$, then the resulting model also satisfies that $\mathfrak{c}=\lambda$. But otherwise, we only know that $\lambda\leq\mathfrak{c}$ in the resulting model.} Then force with random real forcing to increase $\mathfrak{c}$ as large as desired, while keeping $\mathfrak{b}=\mathfrak{d}=\aleph_{\omega+1}$. Both the Hechler iteration (being a finite-support iteration of c.c.c. posets, and hence c.c.c.) and the random real forcing are c.c.c., so by Lemma~\ref{lem:cccPreservesCapturing}, they preserve $\mathrm{cov}(\aleph_\omega)>\aleph_{\omega+1}$.
\end{proof}

While from Gitik's example, Corollary~\ref{cor:noBounds}, we derive this next scenario.
\begin{thm}\label{th:sam_of_1st_fixpt}
    If there are sufficiently large cardinals then it is consistent that for $\kappa$ the first fixed point of the $\aleph$-function we have $\mathfrak{b}=\mathfrak{d}=\aleph_1$,  $\mathrm{cov}(\kappa)$ as large as we like, and $\mathfrak{c}$ even larger.
\end{thm}

\paragraph{Lower Bounds of the Covering Number - Mod Finite  Scales} Below, see Section~\ref{ssec:ksam}, we will want to control the covering number not just of a single cardinal but of families. This will be achieved via lower bounds on the covering number provided by mod finite scales. 

Call a set $A$ of cardinals \emph{suitable} for a cardinal $\kappa$ of countable cofinality if $A$ is countable, every member of $A$ is uncountable, regular and $\sup A=\kappa$.
 Let $<^*$ be the mod finite order on $\prod A$, so $f<^*g$ if and only if for all but finitely many $\mu$ we have $f(\mu)<g(\mu)$.

 \begin{lem}\label{l:suitable}
     If $A$ is suitable for $\kappa$ then $([\kappa]^\omega,\subseteq) \tq (\prod A,<^*)$. 

     Hence, $\mathrm{cov}(\kappa) \ge \sup \{\cof\left(\prod A,<^*\right) :A \text{ suitable for }\kappa\}$.
 \end{lem}
\begin{proof}
    For $S$ in $[\kappa]^\omega$ and $\mu$ in $A$, set $\phi_+(S)(\mu)=\sup (S \cap \mu)+1$, and note, since $S$ is countable and $\mu$ has uncountable cofinality, that this is in $\mu$; so $\phi_+(S)$ is in $\prod A$. For $f$ in $\prod A$ set $\psi_-(f)=\mathrm{im}(f)$. 
    Now if $\psi_-(f) \subseteq S$ then, for all $\mu \in A$, we have $f(\mu)$ in $S$, so $\phi_+(f)(\mu) > \sup (S \cap \mu) \ge f(\mu)$; so $f <^* \phi_+(S)$, as required.
\end{proof}

For sufficiently small $\kappa$ the above lower bound on the covering number is attained. Suppose $\kappa=\aleph_\alpha$ is such that $1 \le \alpha < \omega_1$ and $\alpha$ has cofinality $\omega$.  Then $\kappa$ has a maximum suitable set $A_\infty$, namely $A_\infty=[\aleph_1,\kappa] \cap \textsc{REG}$. By a standard \textsc{PCF} fact (see \cite{AbrahamMagidor}, Theorem 4.4), $\mathrm{cov}(\kappa)=\mathrm{max\  pcf}(A_\infty)=\mathrm{cof}(\prod A_\infty,<)$; this in turn equals $\cof(\prod A_\infty,<^*)$. Thus we have shown:
\begin{lem}
    If $\kappa=\aleph_\alpha$ has  countable cofinality and $1 \le \alpha < \omega_1$ then 
    \[ \mathrm{cov}(\kappa) = \sup \{\cof\left(\prod A,<^*\right) :A \text{ suitable for }\kappa\}.\]
\end{lem}
 
 A (mod finite) \emph{scale} in $\prod A$ is a sequence $\vec{f}=\la f_\alpha:\alpha<\lambda\ra$ of functions in the product which is increasing mod finite ($\alpha<\beta<\lambda$ implies $f_\alpha<^*f_\beta$) and cofinal mod finite. 
 The cardinal $\lambda$ is called the length of the scale. 
 In Tukey terms, there is a scale of length $\lambda$ in $\prod A$ if and only if $\left(\prod A,<^*\right) \te \lambda$. 
If a product $\prod A$ carries a scale of length $\lambda$, then its cofinality  is $\lambda$. 
 Abusing terminology, let us say that $\kappa$ admits a (mod finite) scale of length $\lambda$ if for some $A$ suitable for $\kappa$ the product $\prod A$ carries a scale of length $\lambda$. 
 By Lemma~\ref{l:suitable}, we conclude:

\begin{cor}\label{cor:ScalesIncreaseSam}
    Suppose that $\kappa$ is a singular cardinal of countable cofinality. If $\kappa$ admits a scale of regular length $\lambda$ then $\mathrm{cov}(\kappa)\geq\lambda$. 
\end{cor}

\subsection{The Sampling Number}\label{sec:sam}

Let $\ii$ be the general relation ($\ii$ for ``infinite intersection'') defined as follows: $A  \ii   B$ if and only if $A\cap B$ is infinite. Define the \emph{sampling number} of a cardinal $\kappa$ to be $\mathrm{sam}(\kappa)=\cof( [\kappa]^{\omega},\ii)$. That is, $\mathrm{sam}(\kappa)$ is the least number of countable subsets of $\kappa$ needed to have infinite intersection with every countably-infinite subset of $\kappa$. Note that $\mathrm{sam}(\kappa)\geq\kappa$.

Clearly the sampling is closely related to the covering number. 
The next result, due to Will Brian \cite{Brian}, tells us that in fact they are equal.
\begin{thm}[Brian] \label{th:cov=sam} For every infinite cardinal $\kappa$ we have 
\[([\kappa]^\omega,\ii)\te([\kappa]^\omega,\subseteq), \quad \text{ and hence }\mathrm{sam}(\kappa)=\mathrm{cov}(\kappa).\]
\end{thm}
\begin{proof} First note that a cofinal set for $([\kappa]^\omega,\subseteq)$ is certainly cofinal in $( [\kappa]^{\omega},\ii)$, so taking $\phi_+$ to be the identity shows $([\kappa]^\omega,\subseteq) \tq ( [\kappa]^{\omega},\ii)$. 

For the converse,  note that $\kappa$ and $[\kappa]^{<\omega}$ are bijective, so $([\kappa]^\omega,\ii)$ is isomorphic to $([[\kappa]^{<\omega}]^\omega,\ii)$. We show $([[\kappa]^{<\omega}]^\omega,\ii)\tq([\kappa]^\omega,\subseteq)$. 
    To see this, first for each $T$ in $[\kappa]^\omega$, fix an enumeration, $T=\{t_i\}_{i\in \omega}$, 
    and then  define  $\psi_-:[\kappa]^\omega \to [[\kappa]^{<\omega}]^\omega$ 
    by $\psi_-(\{t_i\}_{i \in \omega})=\{ \{t_i\}_{i  \le n}\} : n \in \omega\}$, and $\phi_+:[[\kappa]^{<\omega}]^\omega \to [\kappa]^\omega$ by $\phi_+(S)=\bigcup S$. 
    Now if $\psi_-(\{t_i\}_{i \in \omega}) \ii S$, then for infinitely many $n$, $\{t_i\}_{i \le n} \in S$, so $t_0,\ldots,t_n \in \bigcup S=\phi_+(S)$, and hence, $\{t_i\}_{i \in \omega} \subseteq \phi_+(S)$, as required. 
\end{proof}

\section{The Minimum Size of Generating Families} 
\subsection{How Many Sequences Suffice?}\label{ssec:seqsam}
\paragraph{The $\mathbf{seq}(M)$ Relation} 
Let $M$ be a (separable metrizable) space, and let $\nci$ be the general relation defined for subsets $S$ and $T$ of $M$ by  
$S \nci T$ if and only if $S \cap T$ is not closed in $T$. 
Set $\mathop{NC}(M)$ to be all subsets of $M$ that are not closed, and let   
 $\mathop{CS^+}(M)$ be all infinite convergent sequences with limit (subsets of $M$ homeomorphic to the subspace $\{0\}\cup\{1/n:n \in \mathbb{N}\}$ of the reals). 
Define $\mathbf{seq}(M)=(NC(M),\mathop{CS}^+(M),\nci)$,
the \emph{sequential structure} of $M$, and  note that its cofinal sets correspond to generating collections of (infinite) convergent sequences. 
Hence our invariant $\mathrm{seq}(M)$ is the cofinality of the relation $\mathbf{seq}(M)$. From \cite{FG_Shape} we have an explicit, combinatorial description of $\mathbf{seq}(M)$ up to Tukey-equivalence which is summarized in the next paragraph. Below we use this to compute $\mathrm{seq}(M)$. 
We will find that $\mathrm{seq}(M)$ is ranked by the size of $M$: if $|M| < |N|$ then $\mathrm{seq}(M)\le\mathrm{seq}(N)$. Consequently we will compute $\mathrm{seq}(M)$ first for countable spaces, before moving on to the uncountable case.

\paragraph{Characterization and Realization of $\mathbf{seq}(M)$}
For a point $x$ in $M$ (recall that we are working with $M$ that are infinite) fix $U_x$ an open neighborhood of $x$ of minimal cardinality. 
 Let $M_\infty$  be the set of $x$ in $M$ with $|U_x|$ equal to $|M|$. 
Separate the points of $M_\infty$ in two as follows. Let $M_\infty^+$ be all $x$ from $M_\infty$ which have a neighborhood base $(B_n)_n$ such that every $D_n=B_n \setminus B_{n+1}$ has cardinality $|M|$. Let $M_\infty^-$ be $M_\infty \setminus M_\infty^+$.
Then a point $x$ from $M_\infty$ is in $M_\infty^-$ if and only if it has a neighborhood base $(B_n)_n$ such that every $D_n=B_n \setminus B_{n+1}$ has cardinality $<|M|$. 
(If $x$ is in $M_\infty^-$ then select a strictly decreasing neighborhood base at $x$, $(B_n)_n$, so $B_1 \subseteq U_x$; as $x$ is not in $M_\infty^+$, every $B_n\subseteq B_{n+1}$ has size $<|M|$.)
Let $\lambda_M=|M_\infty^+|$ and $\mu_M=|M_\infty^-|$. 
The cardinals $|M|,\lambda_M$ and $\mu_M$ suffice to characterize $\mathbf{seq}(M)$ up to Tukey equivalence.

\begin{thm}\label{thm:seqTukeyEquivalence} (\cite{FG_Shape}, Theorem~4.10) Let $M$ be separable metrizable.  Then $\lambda_M \le |M|$, $\mu_M \le \omega$ and $\mathbf{seq}(M)$ is Tukey equivalent to 
    \[
    \Big(P(|M|) \times (\lambda_M,=)\Big) + \Big(Q(|M|) \times (\mu_M,=)\Big) + S(|M|).
    \]   
\end{thm}

The case when $\lambda_M$ and $\mu_M$ are both equal to zero is an important special case. 
Define a space $M$ to be \emph{locally small} if  every point of $M$ has a neighborhood of size strictly less than $|M|$;  
by the definitions given just before Theorem~\ref{thm:seqTukeyEquivalence}, this is equivalent to saying $\lambda_M=0=\mu_M$. The next lemma shows that the cardinality of such a space has countable cofinality and that such spaces exist of size $\kappa<\mathfrak{c}$ for each $\kappa$ of countable cofinality:

\begin{lem}\label{lem:locSmallCtbCof}
    Suppose that $M$ is a locally small, separable metrizable space. Then $|M|$ has countable cofinality. Additionally, if $\kappa<\mathfrak{c}$ has countable cofinality, then there is a locally small,  separable metrizable space of size $\kappa$.
\end{lem}
\begin{proof}
    To show that $|M|$ has countable cofinality, we will show that we can cover $M$ by countably-many sets, each of size $<|M|$. By definition of locally small, every point $x\in M$ has a neighborhood $U_x$ of size $<|M|$. The space $M$, being separable and metrizable, has a countable base, say $\mathcal{B}$. Thus we can find, for each $x\in X$, a basic open $B_x\in\mathcal{B}$ with $x\in B_x\subseteq U_x$. The set of such $B_x\in\mathcal{B}$ is countable, each of size $<|M|$, and covers $M$.
    
    Now suppose that $\kappa$ has countable cofinality, and let $\la\kappa_n:n\in\omega\ra$ be a strictly increasing sequence of cardinals cofinal in $\kappa$. Construct the desired space $M$ by picking, for each $n\in\omega$, $\kappa_n$-many reals from the interval $(n,n+1)$. Then every point from $M$ has a neighborhood (namely $M\cap(n,n+1)$ for some $n$) with fewer than $|M|$-many points.
\end{proof}

Every Tukey type mentioned in the characterization can be realized as the $\mathbf{seq}(M)$ of a separable metrizable space. 
\begin{thm}\label{th:seq_real}(\cite{FG_Shape}, Theorem~4.11) 

For every cardinal $\kappa \le \mathfrak{c}$ there is a separable metrizable space $M=M(\kappa)$ such that $\mathbf{seq}(M) \gte P(\kappa) \times (\kappa,=)$. 

    For every cardinal $\kappa \le \mathfrak{c}$ of countable cofinality and $\lambda < \kappa$ and $\mu \in \omega+1$ there is a separable metrizable space $N=N(\kappa,\lambda,\mu)$ such that $\mathbf{seq}(N) \gte P(\kappa) \times (\lambda,=) + Q(\kappa) \times (\mu,=) + S(\kappa)$. 
\end{thm}

\paragraph{The Countable Case - Countable Spaces}

Let $M$ be a countably infinite, separable metrizable space. In this case the relation $S(|M|)=S(\omega)$ is trivial, because $P(n)$ is trivial for every $n$ in $\omega$. Recall $\cof(P(\omega))=\mathfrak{b}$. 
Further, $M$ partitions into three subsets: (1) $M\setminus M_\infty$, which are the isolated points of $M$, 
(2) $M_\infty^-$, which are the points of $M$ with a neighborhood homeomorphic to the convergent sequence, 
and (3) the remaining points, $M_\infty^+$. 
Putting the data together we compute $\mathrm{seq}(M)$ in the countably infinite case.

\begin{thm}\label{th:seq_ctble} Let $M$ be a countably infinite, separable metrizable space. Then
\[
\mathrm{seq}(M) = 
\begin{cases}
  0 & \text{ if $M$ is discrete} \\
  n & \text{ if $M$ is the disjoint sum of $n$ $(\le \omega)$ convergent sequences} \\
  \mathfrak{b} & \text{ otherwise.}
\end{cases}
\]    
\end{thm}

\paragraph{The Uncountable Case - Connection to the Sampling Number}

It remains to consider uncountable separable metrizable $M$. The first step is to compute the cofinalities of $P(\kappa)$ and $Q(\kappa)$, for some uncountable cardinal $\kappa$ (soon to be $|M|$), in terms of the sampling number of $\kappa$, $\mathrm{sam}(\kappa)$, and the bounding number, $\mathfrak{b}$.

\begin{thm}\label{th:cof_Pkappa}
Let $\kappa$ be an uncountable cardinal, and $(\kappa_n)_n$ a sequence of infinite cardinals increasing (not necessarily strictly) to $\kappa$. 
Then 
\[\cof((\prod_n [\kappa_n]^{<\omega})_\infty,i_\infty)
= \mathrm{sam}(\kappa)\cdot \mathfrak{b}. 
\]
Hence $\cof(P(\kappa))= \mathrm{sam}(\kappa)\cdot\mathfrak{b}$, and if $\cof(\kappa)=\omega$, then $\cof(Q(\kappa))=\mathrm{sam}(\kappa)\cdot\mathfrak{b}$.    
\end{thm}

\begin{proof}
It suffices to show the following pair of inequalities: 

(1)~$\cof(\left(\prod_n [\kappa_n]^{<\omega}\right)_\infty,i_\infty) \ge \mathrm{sam}(\kappa)\cdot\mathfrak{b}$, and 
(2)~$\mathrm{sam}(\kappa)\cdot\mathfrak{b}  \ge \cof(P(\kappa))$.

\smallskip

For (1), first note that each $\kappa_n \ge \omega$, so $(\left(\prod_n [\kappa_n]^{<\omega}\right)_\infty,i_\infty) \gtq P(\omega) \gte (\omega^\omega,\le_\infty)$, and the last relation has cofinality $\mathfrak{b}$. 
We get (1) by combining this with the fact that  $(\left(\prod_n [\kappa_n]^{<\omega}\right)_\infty,i_\infty) \gtq ([\kappa]^\omega,\ii)$. 
To see the latter, define $\phi_+( (F_n)_n)$ to be $\bigcup_n F_n$ when this set is infinite, and an arbitrary element of $[\kappa]^\omega$ otherwise. 
For any $S$ in $[\kappa]^\omega$, enumerate $S$ as $S=\{s_n : n \in \omega\}$, and set $\psi_-(S)=( \{s_n\})_n$. 
Now if $\psi_-(S) i_\infty (F_n)_n$, then for infinitely many $n$ we have $s_n \in F_n$; it follows $\bigcup_n F_n$ is infinite and meets $S$ in an infinite set, so $S \ii \phi_+((F_n)_n)$, as required.

Now we establish (2). First fix $\mathcal{S}$ cofinal in $([\kappa]^\omega,\ii)$ of size $\mathrm{sam}(\kappa)$. 
Take any $S$ in $\mathcal{S}$, and enumerate it as $S=\{s_n : n \in \omega\}$. Observe that the bijection $n \leftrightarrow s_n$ induces a Tukey equivalence between $( ([S]^{<\omega})^{\omega}_{\infty}, i_\infty)$ and $( ([\omega]^{<\omega})^{\omega}_{\infty}, i_\infty)$; and so both have cofinality $\mathfrak{b}$. Fix $\mathcal{C}_S$ cofinal in $( ([S]^{<\omega})^{\omega}_{\infty}, i_\infty)$ of size $\mathfrak{b}$. 

Let $\mathcal{C}= \bigcup \{ \mathcal{C}_S : S \in \mathcal{S}\} \cup \mathcal{C}_0$, where $\mathcal{C}_0 = \{ (F)_n : \emptyset \ne F \in [\kappa]^{<\omega}\}$. 
Note that $\mathcal{C}$ has size at most $\mathfrak{b} \cdot \mathrm{sam}(\kappa)\cdot\kappa=\mathfrak{b}\cdot\mathrm{sam}(\kappa)$. 
We show $\mathcal{C}$ is cofinal in $P(\kappa)$.

Indeed, take any $(G_n)_n$ in $([\kappa]^{<\omega})^{\omega}_{\infty}$. Let $T=\bigcup_n G_n$. If $T$ is finite, then $(T)_n$ is in $\mathcal{C}_0$ and $(G_n)_n i_\infty (T)_n$. 
So suppose $T$ is infinite. Pick $S$ in the cofinal set $\mathcal{S}$ so that $T'=S \cap T$ is infinite. 
For each $n$ set $G'_n=G_n \cap T'$. Since $T'$ is infinite, for infinitely many $n$ we must have $G'_n$ non-empty, so  $(G'_n)_n$ is in $([S]^{<\omega})^{\omega}_{\infty}$.  
As $\mathcal{C}_S$ is cofinal in $(([S]^{<\omega})^{\omega}_{\infty},i_\infty)$, there is a $(F_n)_n$ in $\mathcal{C}_S$ so that $(G'_n)_n i_\infty (F_n)_n$, and so $(G_n)_n i_\infty (F_n)_n$,
as required for $\mathcal{C}$ to be cofinal.
\end{proof}

Combining the previous theorem with the Tukey characterization of $\mathbf{seq}(M)$, Theorem~\ref{thm:seqTukeyEquivalence} above, we compute $\mathrm{seq}(M)$, for uncountable separable metrizable $M$, in terms of the covering number of the size of $M$ and $\mathfrak{b}$. 
\begin{thm}\label{th:seq_sam}
     Let $M$ be an uncountable separable metrizable space. 
\begin{enumerate}[noitemsep,topsep=1pt]
\item[(i)] If $M$ is not locally small then 
   $\mathrm{seq}(M)= \mathrm{cov}(|M|) \cdot \mathfrak{b}$.
     
\item[(ii)]      If $M$ is locally small then $\mathrm{seq}(M)=\lim_{\mu <|M|} \mathrm{cov}(\mu)\cdot \mathfrak{b}$.    
\end{enumerate}
   \end{thm}

   \begin{proof}
   Let $\kappa=|M|$. Let's first fix some notation relevant for both cases: if $\kappa$ has countable cofinality, then fix a strictly increasing sequence of infinite cardinals, $\la\kappa_n:n<\omega\ra$, with limit $\kappa$. Otherwise let $\kappa_n=\kappa$ for all $n$, and note in this case that $(\left(\prod_n [\kappa_n]^{<\omega}\right)_\infty,i_\infty)= P(\kappa)$.
From Theorem \ref{thm:seqTukeyEquivalence}, we know $\mathbf{seq}(M)$ is Tukey equivalent to 
$$
\Big(P(\kappa) \times (\lambda_M,=)\Big) + \Big(Q(\kappa) \times (\mu_M,=)\Big) + S(\kappa),
$$
where $\lambda_M \le \kappa$ and $\mu_M \le \omega$. Recall that
\begin{enumerate}
    \item When $\kappa$ has uncountable cofinality, $Q(\kappa)$ and $S(\kappa)$ are trivial, and
    \item $M$ is locally small iff $\lambda_M=\mu_M=0$ (see the paragraph after the statement of Theorem \ref{thm:seqTukeyEquivalence}).
\end{enumerate}

We will divide the proof into two main cases, namely, whether $\kappa$ has uncountable or countable cofinality. The latter case has two subcases, based upon whether at least one of $\lambda_M$ or $\mu_M$ are non-zero, or whether both are zero.

Consider, then, the case that $\kappa$ has uncountable cofinality. By Lemma \ref{lem:locSmallCtbCof}, this means that $M$ is not locally small, and accordingly, we will show that $\mathrm{seq}(M)=\mathrm{cov}(\kappa)\cdot\mathfrak{b}$. Using the characterization of \textbf{seq}$(M)$ mentioned earlier in the proof and the fact that $\kappa$ has uncountable cofinality, we know that $\mathrm{seq}(M)=\mathrm{cof}(P(\kappa)\times(\lambda_M,=))$. This equals the product $\mathrm{cf}(P(\kappa))\cdot\mathrm{cf}(\lambda_M,=)$ which in turn (using Theorem \ref{th:cof_Pkappa}) equals $\mathrm{sam}(\kappa)\cdot\mathfrak{b}\cdot\lambda_M$. However, $\mathrm{sam}(\kappa)=\mathrm{cov}(\kappa)$ (Theorem \ref{th:cov=sam}), from which we conclude the desired result.

Now consider the case when $\kappa$ has countable cofinality. We have two subcases, as mentioned at the beginning of the proof. First suppose that $\lambda_M=\mu_M=0$. From (2) above, we see that $M$ is locally small. In this case $\mathbf{seq}(M)$ is Tukey equivalent to $S(\kappa)=\sum_n P(\kappa_n) \times (\kappa_n,=)$. Now 
\begin{multline*}
\mathrm{seq}(M)=\cof(S(\kappa))=\lim_n \cof(P(\kappa_n))\cdot\kappa_n =\lim_n\cof(P(\kappa_n)) \\
=\lim_n\mathrm{sam}(\kappa_n)\cdot\mathfrak{b}=\lim_{\mu <|M|} \mathrm{sam}(\mu)\cdot \mathfrak{b}=\lim_{\mu<|M|}\mathrm{cov}(\mu)\cdot\mathfrak{b};
\end{multline*}
note that the second equality uses the definition of $S(\kappa)$ and the cofinality properties of the sum (see the discussion on Operations on Relations in Subsection 2.1), that the third uses Lemma~\ref{lem:cofatleastkappa}, that the next equality follows from Theorem~\ref{th:cof_Pkappa}, that the second-to-last equality uses the fact that the $\kappa_n$ are cofinal in $\kappa$ and that the $\mathrm{sam}$ operator is monotone (Lemma~\ref{lem:samIncreasing}), and the final equality uses the equality of the covering and sampling numbers.

If $\kappa$ has countable cofinality but at least one of $\lambda_M, \mu_M \ne 0$, then  
$$
P(\kappa)\times (\kappa,=) \gtq \mathbf{seq}(M) \gtq \left(\left(\prod_n [\kappa_n]^{<\omega}\right)_\infty,i_\infty\right).
$$
By Theorem~\ref{th:cof_Pkappa}, the right-most relation has cofinality equal to $\mathrm{sam}(\kappa)\cdot\mathfrak{b}$, which equals $\mathrm{cov}(\kappa)\cdot\mathfrak{b}$ and the left-most relation has cofinality equal to $\mathrm{sam}(\kappa)\cdot\mathfrak{b}\cdot\kappa=\mathrm{sam}(\kappa)\cdot\mathfrak{b}=\mathrm{cov}(\kappa)\cdot\mathfrak{b}$. Thus $\mathrm{seq}(M)$ is exactly $\mathrm{cov}(\kappa)\cdot\mathfrak{b}$ in this case as well.
\end{proof}

\paragraph{Computing $\mathrm{seq}(M)$}

For $\mathrm{seq}(M)$ of countable metrizable $M$, see Theorem~\ref{th:seq_ctble}. Here we summarize our understanding of $\mathrm{seq}(M)$ for uncountable $M$.
\begin{itemize}[topsep=2pt]
\item We can compute $\mathrm{seq}(M)$ from $\mathrm{cov}(|M|)$ and $\mathfrak{b}$: 
 $\mathrm{seq}(M)=\mathrm{cov}(|M|)\cdot\mathfrak{b}$, unless $M$ is locally small, in which case 
$\mathrm{seq}(M)=\lim_{\mu <|M|} \mathrm{cov}(\mu)\cdot \mathfrak{b}$.

\item All possible values of $\mathrm{seq}(M)$ are attained: if $\kappa \le \mathfrak{c}$ is uncountable, then there is an $M(\kappa)$ with  $\mathrm{seq}(M(\kappa))=\mathrm{cov}(\kappa)\cdot \mathfrak{b}$; if in addition, $\kappa$ has countable cofinality, then there is an $N(\kappa)$ with $\mathrm{seq}(N(\kappa))=\lim_{\mu <\kappa} \mathrm{cov}(\mu)\cdot \mathfrak{b}$. This follows \textsl{a fortiori} from Theorem~\ref{th:seq_real}. 

\item The sequentiality number is ranked by the size of the space: if $|M| < |N|$ then $\mathrm{seq}(M) \le \mathrm{seq}(N)$. Further, if $\kappa \le \mathfrak{c}$ is uncountable, then there is exactly one value of $\mathrm{seq}(M)$ for $M$ of size $\kappa$, unless $\cof(\kappa)=\omega$, in which case at least one but possibly two values attained (one for locally small $M$, and a potentially strictly larger value for all other $M$ of size $\kappa$).   
\end{itemize}

To go deeper, observe that two cases arise depending on the relative sizes of $\mathrm{cov}(|M|)$ and $\mathfrak{b}$.
 If $\mathrm{cov}(|M|) \le \mathfrak{b}$ then $\mathrm{seq}(M)=\mathfrak{b}$. 
In particular, if $\mathfrak{b}=\mathfrak{c}$ then $\mathrm{seq}(N)=\mathfrak{c}$ for all uncountable separable metrizable $N$. 

Now suppose $\mathfrak{b} \le \mathrm{cov}(|M|)$. In this case the following hold.
\begin{itemize}[topsep=2pt]
\item If $M$ is not locally small, then $\mathrm{seq}(M)=\mathrm{cov}(|M|)$, since 
$$
\mathrm{seq}(M)=\mathrm{cov}(|M|)\cdot\mathfrak{b}=\mathrm{cov}(|M|).
$$ 
\item If $M$ is locally small, then $\mathrm{seq}(M)\leq\mathrm{cov}(|M|)$ since 
$$
\mathrm{seq}(M)=\lim_{\mu<|M|}\mathrm{cov}(\mu)\cdot\mathfrak{b}\leq\mathrm{cov}(|M|)\cdot\mathfrak{b}=\mathrm{cov}(|M|),
$$
using Lemma \ref{lem:samIncreasing} for the monotonicity of the covering number.
\item Always $\mathrm{seq}(M) \ge |M|$, and   if $|M|=\aleph_n$ or $|M|=\mathfrak{c}$ then $\mathrm{seq}(M)=|M|$. Indeed, if $|M|=\aleph_n$ for some $n$ or if $|M|=\mathfrak{c}$, then $|M|\leq\mathrm{seq}(|M|)\leq\mathrm{cov}(|M|)=|M|$, using Lemmas \ref{l:sam_aleph_n} and \ref{l:sam1}(i), respectively.

\item If $\cof(|M|)=\omega$ but $M$ is not locally small, then $\mathrm{seq}(M) \ge |M|^+$. Indeed, from the first bullet in this list, $\mathrm{seq}(M)=\mathrm{cov}(|M|)\geq |M|^+$, where the inequality follows from Lemma \ref{l:sam1}(ii).

\item If $|M|=\aleph_\alpha$ is strictly below the first fixed point of the $\aleph$-function then $\mathrm{seq}(M) < \aleph_{|\alpha|^{+4}}$. This holds because $\mathrm{seq}(M)\leq\mathrm{cov}(|M|)$ (using the first and second bullets in this list) and by applying Proposition \ref{pr:fixpt}.

\item It is consistent, assuming the existence of sufficiently large cardinals, that there is a separable metrizable $M$ of size the first fixed point of the $\aleph$-function, with $\mathrm{seq}(M)$ as far above $|M|$ as we like. To see this, start with the model from Corollary \ref{cor:noBounds}, where $\kappa$ denotes the first fixed point of the $\aleph$ function and where $\mathrm{cov}(\kappa)$ is arbitrarily large. Force to add $\kappa$-many random reals over this model, preserving that $\mathfrak{b}=\omega_1$ (\cite{Kunen2011} Lemma IV.7.36). Now let $M$ be a subspace of the reals of size $\kappa$ which is not locally small; one can construct such an $M$ by including 0 as well as $\kappa$-many real numbers from each rational interval containing 0. Since $\mathfrak{b}=\omega_1<|M|\leq\mathrm{cov}(|M|)$, the first bullet in this list implies that $\mathrm{seq}(M)=\mathrm{cov}(|M|)$. Consequently, $\mathrm{seq}(M)=\mathrm{cov}(\kappa)$ is arbitrarily large.

\item It is consistent, assuming the existence of a supercompact cardinal, that there is a separable metrizable $M$ with $|M|=\aleph_\omega$ but $\mathrm{seq}(M) > |M|^+$. Here's the argument: start with the model from Theorem \ref{th:small_b_big_sam} where $\mathfrak{b}=\mathfrak{d}=\aleph_1$, $\mathfrak{c}$ is arbitrarily large (in particular, at least $\aleph_{\omega+1}$), and where $\mathrm{cov}(\aleph_\omega)>\aleph_{\omega+1}$. Let $M$ be a subspace of the reals of size $\aleph_\omega$ which is not locally small. Then as in the previous item, $\mathrm{seq}(M)=\mathrm{cov}(|M|)$. Hence $\mathrm{seq}(M)=\mathrm{cov}(\aleph_\omega)>\aleph_{\omega+1}$.

\item It is consistent, assuming the existence of a supercompact cardinal, that there is a separable metrizable $M$ with $|M| = \mathfrak{b}$ but $\mathrm{seq}(M) > \mathfrak{b}$. The argument is similar to the one in the previous bullet, using Theorem \ref{th:sam_of_b} and working with $M$ of size $\aleph_{\omega+1}$ which is not locally small.
\end{itemize}

It is an open problem whether large cardinals are necessary for the last two points. Is it true, under some `no large cardinals' hypothesis, that $\mathrm{seq}(M)=|M|$ if $\cof(|M|)>\omega$ or $M$ is locally small, and otherwise $\mathrm{seq}(M)=|M|^+$?
See Question~\ref{qu:large_card}.

\subsection{How Many Compact Sets Suffice?}\label{ssec:ksam}

\paragraph{The $\mathbf{k}(M)$ Relation}
The \emph{$k$-structure} of a (separable metrizable) space $M$ is the relation $\mathbf{k}(M)=(\mathop{NC}(M),\K(M),\nci)$, where $\mathop{NC}(M)$  are the non-closed subsets of $M$, the general relation $\nci$ is ``not closed in'' and $\K(M)$ is the collection of all compact subsets of $M$.  
Observe  that cofinal sets of $\mathbf{k}(M)$ correspond to generating compact collections, and hence $\mathop{k}(M)=\cof(\mathbf{k}(M))$. Clearly, $\mathbf{seq}(M) \tq \mathbf{k}(M)$, and so $\mathrm{seq}(M) \ge \mathop{k}(M)$. 
Additionally, define the \emph{$kc$-structure}  $\mathbf{kc}(M)=(M,\K(M),\in)$, whose  cofinal sets are the compact covers of $M$,  and the \emph{compact-covering number}, $\mathop{kc}(M)=\cof(\mathbf{kc}(M))$, the minimum number of compact sets needed to cover $M$ (see \cite{FG_Shape_CC}). Note $kc(M)\le \omega$ if and only if $M$ is $\sigma$-compact.
From Lemma~3.5 of \cite{FG_Shape} we know that $\mathbf{k}(M) \tq \mathbf{kc}(M)$, and so $\mathop{k}(M) \ge kc(M)$, for every separable metrizable space, $M$, which is not $\sigma$-compact. 

Our results above on $\mathrm{seq}(M)$ -  building on the characterization of $\mathbf{seq}(M)$ up to Tukey equivalence obtained in \cite{FG_Shape}, along with what is known about the covering number - tell us the following: that $\mathrm{seq}(M)$ is ranked by the size of $M$, if $|M|<|N|$ then $\mathrm{seq}(M)\le\mathrm{seq}(N)$, and for a fixed size of space at most two values of $\mathrm{seq}(M)$ can occur:  \ for every uncountable $\kappa \le \mathfrak{c}$
\[ \{\mathrm{seq}(M) : |M|=\kappa\} \subseteq \left\{\lim_{\mu <|M|} \mathrm{cov}(\mu) \cdot\mathfrak{b},\mathrm{cov}(|M|)\cdot\mathfrak{b}\right\}.\]

The results of \cite{FG_Shape} give in complete detail the Tukey types of $\mathbf{k}(M)$ for $M$ $\sigma$-compact, from which we deduce the corresponding values of $k(M)$ in the next paragraph. We turn then to the case when $kc(M)$ is uncountable,  i.e., when $M$ is not $\sigma$-compact. 
In contrast to the $\sigma$-compact case, simple descriptions of the Tukey type of $\mathbf{k}(M)$ are not known for spaces which are not $\sigma$-compact. To the contrary, from \cite{FG_Shape} (see Theorem 3.13) we know that among $M$ which are not $\sigma$-compact, the structure of Tukey types of $\mathbf{k}(M)$'s is immensely complex. There is a Tukey order anti-chain of size $2^\mathfrak{c}$, and both $\mathfrak{c}^+$ and the unit interval (with the usual order) order embed in $\mathbf{k}(M)$'s under the Tukey order.

Nevertheless we show below that $kc(M)$ plays a role for $k(M)$ loosely analogous to the role that $|M|$ plays for $\mathrm{seq}(M)$, with the connection being made by the covering number combined with the bounding number. 
Theorem~\ref{th:k_sam} below implies that:  \ for every uncountable $\kappa \le \mathfrak{c}$
\[\{\mathop{k}(M) : kc(M)=\kappa\} \subseteq [\kappa\cdot\mathfrak{b},\mathrm{cov}(\kappa)\cdot\mathfrak{b}].\]

Above we explored, in the context of $\mathrm{seq}(M)$, the possible values of $\mathrm{cov}(\kappa)$ and the gap between $\kappa$ and $\mathrm{cov}(\kappa)$, for $\kappa=|M|$. We do not rehash here the situation when $\kappa=kc(M)$. Instead in the sequel we focus on what is different between  $\mathrm{seq}(M)$ and $k(M)$, 
namely (1)  that while $\mathrm{seq}(M)$  for $|M|=\kappa$ can only take on the value of one of a lower and an upper bound,  
$\{\mathop{k}(M) : kc(M)=\kappa\}$ is constrained merely to an \emph{interval},  and (2) the lower bound ($\kappa\cdot\mathfrak{b}$) on this interval of $k(M)$ values is potentially smaller than the analogous lower bound value for $\mathrm{seq}(M)$ values ($\lim_{\mu <\kappa} \mathrm{cov}(\mu)\cdot\mathfrak{b}$). 
Our task is made more difficult because we have no analog to the realization of types theorem, Theorem~\ref{th:seq_real}, for $\mathbf{seq}(M)$. 

Let $\kappa$ be an uncountable cardinal  with $\kappa \le \mathfrak{c}$.  Define the \emph{$k$-spread} at $\kappa$, to be $\sigma_k(\kappa)=|\{k(M) : kc(M)=\kappa\}|$. We show in Theorem~\ref{th:k=sam} that 
the upper bound of the interval $[\kappa\cdot\mathfrak{b},\mathrm{cov}(\kappa)\cdot\mathfrak{b}]$ is always attained; that is, for some $M$ we have $kc(M)=\kappa$ and $k(M)=\mathrm{cov}(\kappa)\cdot\mathfrak{b}$, and hence $\sigma_k(\kappa) \ge 1$. 
We then show, Theorem~\ref{th:powers}, that in certain circumstances the lower bound, $\kappa\cdot\mathfrak{b}$, is also attained. It follows that $\sigma_k(\kappa)$ may be at least two. 
Finally we show below that, modulo a conjecture about the configurations of certain covering numbers, for every $2 \le K < \omega$ it is consistent that for some $\kappa$ we have 
$\sigma_k(\kappa)\geq K$ (Theorem~\ref{th:K_many}),  and modulo yet stronger conjectures about covering numbers, it is consistent that for some $\kappa$ we have 
$\sigma_k(\kappa)\ge \aleph_0$ (Theorem~\ref{th:ctbl_many}).

\paragraph{The Countable Case - $\sigma$-Compact Spaces}

Van Douwen computed $\mathop{k}(M)$ for $\sigma$-compact, metrizable $M$ (note that such an $M$ is separable). We note that this data can also be read off from the classification of Tukey types of $\mathbf{k}(M)$ in \cite{FG_Shape} for these $M$. 

\begin{thm}\label{thm:VanDouwenKness}
    Let $M$ be $\sigma$-compact and metrizable. Then
\[
\mathop{k}(M) = 
\begin{cases}
  0 & \text{ if $M$ is discrete} \\
  1 & \text{ if $M$ is compact} \\
  \omega & \text{ if $M$ is locally compact but not compact} \\ 
  \mathfrak{b} & \text{ otherwise.}
\end{cases}
\]    
\end{thm}

\paragraph{Another Connection To Sampling}

We now see that for non $\sigma$-compact separable metrizable spaces the $k$-ness number is connected to the compact-covering number via the covering number.

\begin{thm}\label{th:k_sam}
   For every separable metrizable $M$ which is not $\sigma$-compact 
   \[\mathrm{cov}(kc(M))\cdot\mathfrak{b} \ge k(M) \ge kc(M)\cdot\mathfrak{b}.\]
   Consequently, for every uncountable $\kappa\leq\mathfrak{c}$,
   \[\sigma_k(\kappa):=
   \{\mathop{k}(M) : kc(M)=\kappa\} \subseteq [\kappa\cdot\mathfrak{b},\mathrm{cov}(\kappa)\cdot\mathfrak{b}].
   \]
\end{thm}
\begin{proof} Since $M$ is not $\sigma$-compact we know, by Lemma~3.5 and Theorem~3.7(4) of \cite{FG_Shape}, $\mathbf{k}(M) \gtq \mathbf{kc}(M)$ and $\mathbf{k}(M) >_T (\omom,\le_\infty)$, so $k(M) \ge kc(M)\cdot\mathfrak{b}$.

Now let $\kappa=kc(M)$, and fix  $\mathcal{K}=\{K_\alpha : \alpha \in \kappa\}$  a compact cover of $M$. We will show that $(\kappa,=) \times (\left( [\kappa]^{<\omega} \right)^{\omega}_{\infty},i_\infty) \tq \mathbf{k}(M)$. By Theorem~\ref{th:cof_Pkappa}, this is enough to argue that $\mathrm{sam}(kc(M))\cdot\mathfrak{b} \ge k(M)$: indeed, the cofinality of the product $(\kappa,=) \times (\left( [\kappa]^{<\omega} \right)^{\omega}_{\infty},i_\infty)$ is equal to the product of the cofinalities, namely $\kappa\cdot\mathrm{sam}(\kappa)\cdot\mathfrak{b}$ (using the cited theorem for the second factor), and since $\mathrm{sam}(\kappa)\geq\kappa$, we may rewrite the cofinality of the product as $\mathrm{sam}(\kappa)\cdot\mathfrak{b}$.

    For each $x$ in $M$, fix $\alpha(x)$ in $\kappa$ so that $x \in K_{\alpha(x)}$. For any finite subset $F$ of $\kappa$, let $K_F$ be the compact set $K_F:=\bigcup_{\alpha \in F} K_\alpha$. 
Fix a metrizable compactification $\gamma M$ of $M$, with compatible metric $d$, and select a base $\mathcal{B}=\{B_m : m \ge 1\}$ for $\gamma M$ where $B_1=\gamma M$, $\mathop{diam} B_m \to 0$ and for every $x$ in $\gamma M$ the set $\{n : x \in B_n\}$ is infinite.

    Define $\phi_+(\alpha, (F_n)_n)$ to be $K_\alpha \cup \bigcup \{ K_{F_m} \cap \cl{B_m} : m \ge 1 \,\,\&\,\, K_\alpha \cap \cl{B_m} \ne \emptyset\}$. 
We check  $\phi_+(\alpha,(F_n)_n)$ is compact. 
Take any cover by basic open sets. Some finite subcollection, say $\mathcal{U}_1$, covers the compact set $K_\alpha$. 
As $K_\alpha$ and $\gamma M \setminus \bigcup \mathcal{U}_1$ are compact and disjoint, the distance between them is strictly positive. 
As $\mathop{diam} B_m \to 0$, it follows that,  of the $\cl{B_m}$ that meet $K_\alpha$, all but finitely many  are contained in $\bigcup \mathcal{U}_1$. 
Hence $\phi_+(\alpha,(F_n)_n) \setminus \bigcup  \mathcal{U}_1$ is compact, and can be covered with a finite subcollection, say $\mathcal{U}_2$, from the basic open cover. Now $\mathcal{U}_1 \cup \mathcal{U}_2$ is a finite subcover of $\phi_+(\alpha,(F_n)_n)$.
    
Having defined $\phi_+$, we now work to define $\psi_-$. Let $N$ be in $NC(M)$. 
    Pick $S=(x_n)_n$ an infinite sequence of points from $N$ converging, say, to $x\notin N$. Let $\alpha=\alpha(x)$. Extract a subsequence from the  base giving a local base at $x$, $(B_{m_n})_n$, such that:  $(m_n)_n$ is strictly increasing, $m_1=1$,  $\cl{B_{m_{n+1}}} \subseteq B_{m_n}$ for all $n$, and $S \cap (\cl{B_{m_n}} \setminus B_{m_{n+1}})$ is non-empty for all $n$. Note also that $S \cap (\cl{B_{m_n}} \setminus B_{m_{n+1}})$ is finite, since otherwise $\cl{B_{m_n}} \setminus B_{m_{n+1}}$ would contain $x$, the limit of $S$. 
    Set $G_{m_n}=\{\alpha(x') : x' \in S \cap (\cl{B_{m_n}} \setminus B_{m_{n+1}})\}$, and by the preceding comments,  this set is finite and non-empty. 
    For all other $m$ set $G_m=\emptyset$. Then define $\psi_-(N) = (\alpha,(G_m)_m)$.

    We verify that $(\psi_-,\phi_+)$ are a Tukey morphism pair. 
    Suppose that 
    $
    \psi_-(N) (=\times\ii) (\alpha,(F_m)_m)$. 
    Then $\alpha=\alpha(x)$, so $x$ is in $K_\alpha$; further for infinitely many $m$, $G_m$ meets $F_m$. 
    Since $G_m=\emptyset$ when $m$ is not an $m_n$, it follows, for infinitely many $n$, $G_{m_n}$ meets $F_{m_n}$. Thus, looking at the definition of $\phi_+(\alpha,(F_m)_m)$, and choice of $(B_{m_n})_n$, we see $S \cap \phi_+(\alpha,(F_m)_m)$ is infinite, and so not closed in $\phi_+(\alpha,(F_m)_m)$. \textsl{A fortiori}, $N \cap \phi_+(\alpha,(F_m)_m)$ is not closed in $\phi_+(\alpha,(F_m)_m)$, as required.
\end{proof}

\paragraph{Totally Imperfect - The Upper Bound is Always Attained}

Recall that a separable metrizable space $M$ is \emph{totally imperfect} if every compact subset of $M$ is countable. Since uncountable compact metrizable spaces contain a Cantor set, $M$ being totally imperfect is equivalent to saying that $M$ contains no  compact subsets of size continuum. We record here some facts about totally imperfect spaces that we will use throughout the rest of the paper:
\begin{enumerate}
    \item If $M$ has size strictly less than the continuum, then $M$ is totally imperfect.
    \item Locally small spaces are totally imperfect, since they have size $<\mathfrak{c}$: indeed, $|M|\leq\mathfrak{c}$ and from Lemma \ref{lem:locSmallCtbCof}, we know that $\mathrm{cof}(|M|)=\omega$ whereas $\mathrm{cof}(\mathfrak{c})>\omega$.
    \item If $M$ is totally imperfect, then kc$(M)=|M|$: since every compact subset of $M$ is countable, one needs $|M|$-many compact sets to cover $M$.
    \item Finally, there are totally imperfect separable metrizable spaces of size continuum, Bernstein subsets of the reals, for example.
\end{enumerate}

Crucially, as stated in the next result, we can determine the $k$-ness number of totally imperfect spaces exactly in terms of $\mathrm{cov}(kc(M))$ and the bounding number.

\begin{thm}\label{th:k=sam}
Let $M$ be a non $\sigma$-compact, separable metrizable space.
\begin{enumerate}[noitemsep,topsep=1pt]
\item[(i)] If $M$ is totally imperfect but not locally small then 
$k(M)=\mathrm{cov}(kc(M))\cdot\mathfrak{b}$.

\item[(ii)] If $M$ is locally small then 
$
k(M)=\lim_{\mu < kc(M)} \mathrm{cov}(\mu)\cdot\mathfrak{b}$. 
\end{enumerate}
\end{thm}
\begin{proof}
We begin with item (i). Since $M$ is not $\sigma$-compact, Theorem \ref{th:k_sam} implies that $\mathrm{sam}(kc(M))\cdot\mathfrak{b} \ge k(M) \ge kc(M)\cdot\mathfrak{b}$. Since $\mathrm{sam}(kc(M))=\mathrm{cov}(kc(M))$ by Theorem \ref{th:cov=sam}, we conclude that $k(M)\leq\mathrm{cov}(kc(M))\cdot\mathfrak{b}$, proving one of the inequalities in (i). For the other inequality, namely that $k(M)\geq\mathrm{cov}(kc(M))\cdot\mathfrak{b}$, we will first show that it suffices to prove that $k(M)\geq\mathrm{seq}(M)$. To begin we use the inequalities at the beginning of the paragraph to conclude that $k(M)\geq\mathfrak{b}$. Next, since $M$ is totally imperfect, $kc(M)=|M|$ (see the remarks before the statement of the theorem), and consequently $\mathrm{cov}(kc(M))=\mathrm{cov}(|M|)$. Furthermore, $M$ is not locally small, so by Theorem~\ref{th:seq_sam}, we have $\mathrm{seq}(M)=\mathrm{sam}(|M|)\cdot\mathfrak{b}$. From what we've just observed, this equals $\mathrm{cov}(kc(M))\cdot\mathfrak{b}\geq\mathrm{cov}(kc(M))$. Therefore, to finish this second inequality in the proof of (i), it suffices to show that $k(M)\geq\mathrm{seq}(M)$.

Let $\K$ be a family of compact subsets of $M$ which has size $k(M)$ and is generating. Each $K$ in $\K$ is countable, and so by Theorem~\ref{th:seq_ctble}, $\mathrm{seq}(K) \le \mathfrak{b}$. Accordingly, fix $\mathcal{S}_K$ which has size no more than $\mathfrak{b}$ and is cofinal in $\mathbf{seq}(K)$. 
We claim that $\mathcal{S}=\bigcup \{ \mathcal{S}_K : K \in \K\}$, which is a subset of $CS^+(M)$ of size no more than $k(M)\cdot\mathfrak{b}$, is cofinal in $\mathbf{seq}(M)$; note that $k(M)\cdot\mathfrak{b}=k(M)$, as remarked in the first paragraph. To see that $\mathcal{S}$ is cofinal, take any $N$ in $NC(M)$. There is a $K$ in $\K$ so that $N'=N \cap K$ is not closed in $K$. Then there is an $S$ in $\mathcal{S}_K$ with $N' \cap S$ not closed in $S$. Now $S$ is in $\mathcal{S}$ and $N\cap S =N'\cap S$ is not closed in $S$, as required.

\smallskip

Now for (ii). Let $M$ be locally small. Then, by definition of `locally small', the Lindel\"{o}f property and the regularity of $M$, there is a countable open cover, $\mathcal{U}=\mathcal{U}_M$, such that for every $U$ from $\mathcal{U}$, we have $|\cl{U}| < |M|$. Since $\cal{U}$ is countable, we then obtain 
$
|M|=(\sup_{} \{|\cl{U}|: U \in \mathcal{U}\})\cdot\aleph_0=\sup\{|\cl{U}|:U \in \mathcal{U}\}$,
and hence $\cof(|M|)=\omega$. Since $M$ is a separable metrizable space, $|M|\leq\mathfrak{c}$, and as $\mathfrak{c}$ has uncountable cofinality, we conclude that $|M|<\mathfrak{c}$. It then follows that $M$ is totally imperfect; as in the first paragraph, this in turn implies that $kc(M)=|M|$. 
We need to show, for such a space, both of (a)${}_M$  $k(M) \ge \lim_{\mu < |M|} \mathrm{sam}(\mu) \cdot\mathfrak{b}$ and (b)${}_M$ $k(M) \le \lim_{\mu < |M|} \mathrm{sam}(\mu)\cdot\mathfrak{b}$.

Towards (a)${}_M$, suppose, inductively that $M$ is locally small, and that for every locally small $M'$ of strictly smaller size, (a)${}_{M'}$ holds. 
Take any $\mu < |M|$, and we will show $k(M) \ge \mathrm{sam}(\mu)\cdot\mathfrak{b}$. 
Let $\mathcal{U}=\mathcal{U}_M$, as above, and pick $U$ in $\mathcal{U}$ such that $\mu < |\cl{U}| < |M|$. 
Set $M'=\cl{U}$. As $M'$ is a closed subset of $M$, $k(M) \ge k(M')$. 
If $M'$ is not locally small, then by part (i), $k(M') = \mathrm{sam}(|M'|)\cdot\mathfrak{b} \ge \mathrm{sam}(\mu)\cdot\mathfrak{b}$, as required. 
Otherwise $M'$ is locally small, and then by the inductive hypothesis, $k(M') \ge \lim_{\mu' < |M'|} \mathrm{sam}(\mu')\cdot\mathfrak{b} \ge \mathrm{sam}(\mu)\cdot\mathfrak{b}$ (as $\mu < |M'|$), again as required for (a)${}_M$.

To complete the proof of (ii), we verify (b)${}_M$. 
Let $\mathcal{U}=\mathcal{U}_M$, as above. For each $U$ from $\mathcal{U}$, $\cl{U}$ is totally imperfect, since $M$ is. Once more arguing as in the first paragraph, we conclude that $kc(\cl{U})=|\cl{U}|$. We thus know by Theorem~\ref{th:k_sam} (in the case that $\bar{U}$ is not $\sigma$-compact) and Theorem~\ref{thm:VanDouwenKness} (in the case that $\bar{U}$ is $\sigma$-compact) that $k(\cl{U}) \le \mathrm{sam}(|\cl{U}|)\cdot\mathfrak{b}$, so fix $\mathcal{C}_U$ cofinal in $\mathbf{k}(\cl{U})$ of size no more than $\mathrm{sam}(|\cl{U}|)\cdot\mathfrak{b}$. 
Set $\mathcal{C}=\bigcup \{ \mathcal{C}_U : U \in \mathcal{U}\}$, and note that $|\mathcal{C}| \le \lim_{\mu < |M|} \mathrm{sam}(\mu) \cdot\mathfrak{b}$. We check $\mathcal{C}$ is cofinal in $\mathbf{k}(M)$, and so confirm (b)${}_M$. 
Take any $N$ in $NC(M)$. Pick a sequence $S$ on $N$ converging to a point $x$ outside $N$. Then $x$ is in some $U$ from the cover $\mathcal{U}$, and - dropping finitely many terms if needed - we can suppose $S$ is contained in $U$. Now some member of $\mathcal{C}_U$, and so $\mathcal{C}$,  meets $S$ infinitely often, as required for cofinality.
\end{proof}

It follows from (i) that the top of the range, $\mathrm{cov}(\kappa)\cdot\mathfrak{b}$, is always obtained, in \textsc{ZFC}.
\begin{cor}\label{cor:topattained}
    If $\omega < \kappa \le \mathfrak{c}$, then there is a separable metrizable space $M_\kappa$ so that $k(M_\kappa)=\mathrm{cov}(\kappa)\cdot\mathfrak{b}$.
\end{cor}
To see this, if $\kappa<\mathfrak{c}$ then let $M_\kappa$ be of size $\kappa$, but not locally small, and otherwise  let $M_\mathfrak{c}$ be totally imperfect of size $\mathfrak{c}$ (a Bernstein set, for example).

Similarly, from (ii) (taking $M_\kappa$ to be locally small of size $\kappa$ when $\kappa<\mathfrak{c}$ has countable cofinality; see Lemma \ref{lem:locSmallCtbCof}) we have:
\begin{cor}
    If $\omega < \kappa \le \mathfrak{c}$ and $\cof(\kappa)=\omega$, then there is a separable metrizable space $M_\kappa$ so that $k(M_\kappa)=\lim_{\mu < \kappa} \mathrm{cov}(\mu)\cdot\mathfrak{b}$.
\end{cor}
This may, or may not, place an additional value in $\{k(M) : \mathop{kc}(M)=\kappa\}$ - depending on whether $\lim_{\mu < \kappa} \mathrm{cov}(\mu)\cdot\mathfrak{b}$ and $\mathrm{cov}(\kappa)\cdot\mathfrak{b}$ are equal.

\paragraph{Powers - Attaining the Lower Bound} 

We next give a particular case when $k(M)$ realizes the lowest possible value, $kc(M)\cdot\mathfrak{b}$. 
\begin{thm}\label{th:powers} 

    Let $N$ be separable metrizable but not compact, and set $M=N^\omega$. Then $kc(M) \ge \mathfrak{d}$ and $k(M)=kc(M)$.     
    Further, if $|N|<\mathfrak{c}$ then $k(M)=kc(M)=\mathrm{cov}(|N|)\cdot \mathfrak{d}$. 
\end{thm}
\begin{proof} 
    Since $N$ is not compact it contains a closed copy of $\omega$. Hence $M$ contains a closed copy of $\omega^\omega$, and so $kc(M) \ge kc(\omega^\omega)=\mathfrak{d}$. 
Again since $N$ is not compact, $M$ is not $\sigma$-compact, so by Theorem \ref{th:k_sam}, $k(M) \ge kc(M)$. 
    We know from Theorem~3.6 of \cite{FG_Shape} that $\mathbf{kc}(M^\omega)\te\mathbf{kc}(M)^\omega \tq \mathbf{k}(M)$, so $kc(M^\omega) \ge k(M)$. 
    But $M^\omega$ is homeomorphic to $M$, from which it follows that $k(M)=kc(M)$.

    Now suppose $|N|<\mathfrak{c}$, so compact subsets of $N$ are countable and $kc(N)=|N|$. We next show that $kc(M) \ge \mathrm{cov}(|N|)$ by verifying $\textbf{kc}(M) \tq ([N]^\omega,\subseteq)$. 
    For $K$ in $\K(N^\omega)$ set $\phi_+(K)$ equal to $\bigcup \{ \pi_n(K) : n \in \omega\}$ if this latter set is infinite (and an arbitrary infinite subset of $N$ otherwise). 
    For $S=\{s_n : n \in \omega\}$ an infinite subset of $N$, set  $\psi_-(S)=\lb (s_n)_n\rb$. Then clearly, if $\psi_-(S) \in K$ then for every $n$, $s_n \in \pi_n(K)$, so $S \subseteq \phi_+(K)$. 

    To complete the argument we show $\mathrm{cov}(|N|)\cdot\mathfrak{d} \ge kc(M)$. Let $\mathcal{C}$ be cofinal in $([N]^\omega,\subseteq)$ of size $\mathrm{cov}(|N|)$. For each $C$ in $\mathcal{C}$ fix a bijective enumeration, $C=\{c_0,c_1,\ldots, c_n\ldots\}$. 
    Let $\mathcal{F}$ be cofinal in $(\omega^\omega,\le)$ of size $\mathfrak{d}$. 
    For each $C$ in $\mathcal{C}$ and $f$ from $\mathcal{F}$, define $K_{C,f}=\prod_n K_{C,f,n}$ where $K_{C,f,n}=\{c_0,c_1,\ldots,c_{f(n)}\}$. Note $K_{C,f}$ is a compact subset of $M=N^\omega$. 
    We check that the set $\{K_{C,f} : C \in \mathcal{C}, f \in \mathcal{F}\}$, which has size no more than $\mathrm{cov}(|N|)\cdot\mathfrak{d}$, covers $M$. 

    Take any point $(x_n)_n$ in $N^\omega$. 
    Pick $C$ in $\mathcal{C}$ so that $\{x_n : n \in \omega\} \subseteq C=\{c_0,c_1,\ldots,c_i,\ldots\}$. For each $n$ let $g(n)$ be such that $x_n=c_{g(n)}$. Pick $f$ in $\mathcal{F}$ so that $g \le f$. Now, for each $n$, $x_n$ is in $K_{C,f,n}$. Hence $(x_n)_n$ is in $K_{C,f}$, as required. 
\end{proof}

A natural question is whether, in the above scenario where $M=N^\omega$, they have the same $k$-ness numbers: if $N$ is non-compact, separable metrizable then is $k(N^\omega)=k(N)$? 

\paragraph{The Number of Values Attained}

Recall from the opening paragraphs of Section \ref{ssec:ksam} that the $k$-spread at an uncountable $\kappa\le \mathfrak{c}$ is $\sigma_k(\kappa)=|\{k(M) : kc(M)=\kappa\}|$, namely, the number of values of $k$ attained for a fixed $kc$. By Theorem \ref{th:k_sam}, the $k$-spread is contained in the interval $[kc(M)\cdot\mathfrak{b},\mathrm{cov}(kc(M))\cdot\mathfrak{b}]$. By Corollary \ref{cor:topattained}, the top of this interval is always attained (and in particular, $\sigma_k(\kappa)\geq 1$ for all uncountable $\kappa\leq\mathfrak{c}$).
 
Further, for every $n\in \mathbb{N}$ we know $\mathrm{cov}(\aleph_n)=\aleph_n$, so $\sigma_k(\aleph_n)=1$. We now show that, modulo conjectures about the patterns of certain covering numbers, we can build models in which $\sigma_k(\kappa)$ is at least $N \geq 2$ and models in which $\sigma_k(\kappa)$ is infinite. We will also briefly discuss what it would take to obtain these configurations.

\begin{conj}\label{conj:finite}
    For each integer $N\geq 2$, it is consistent that there exists a sequence $\kappa_1,\dots,\kappa_N$ of singular cardinals of cofinality $\omega$ so that 
    $$
    \mathfrak{b}<\kappa_1<\dots< \kappa_N<\mathrm{cov}(\kappa_1)<\dots<\mathrm{cov}(\kappa_N).
    $$
\end{conj}

\begin{conj}\label{conj:infinite}
    There exists a strictly increasing sequence $\la\kappa_n:n\geq 1\ra$ of singular cardinals of cofinality $\omega$ so that
    \begin{enumerate}
        \item $\mathfrak{b}<\kappa_1$
        \item $\la\mathrm{cov}(\kappa_n):n\geq 1\ra$ is strictly increasing
        \item $(\sup_{n\geq 1}\kappa_n)<\mathrm{cov}(\kappa_1)$.
    \end{enumerate}
     
\end{conj}

We now build up to a lemma that gives some consequences of Conjecture \ref{conj:finite} and perforce of Conjecture \ref{conj:infinite}. To do so, we need to first briefly review a specific instance of the definition of the pseudopower (see \cite{Eisworth} for a detailed discussion) and a few standard facts. For $\kappa$ of cofinality $\omega$, we define $\mathrm{pp}(\kappa)$ to be the sup across all cofinal $\mathfrak{a}\subset\kappa$ of countable cardinality and across all ideals $J$ extending the bounded ideal on $\mathfrak{a}$ of the true cofinality of $\left(\prod\mathfrak{a},<_J\right)$ (if the true cofinality exists).

For the following discussion, we need to know that for all cardinals $\kappa>\omega$ of cofinality $\omega$, $\mathrm{pp}(\kappa)\leq\mathrm{cov}(\kappa)$. Moreover, the pseudopower satisfies an inverse monotonicity condition: if $\kappa<\lambda$ are singular cardinals of cofinality $\omega$ and if $\lambda\leq\mathrm{pp}(\kappa)$, then $\mathrm{pp}(\lambda)\leq\mathrm{pp}(\kappa)$.

Now we are ready for the following lemma on the consequences of the above conjectures.

\begin{lem}\label{lem:consequences}
    Suppose that Conjecture \ref{conj:finite} is consistent, and work in a model with 
    $$
    \mathfrak{b}<\kappa_1<\dots<\kappa_N<\mathrm{cov}(\kappa_1)<\dots<\mathrm{cov}(\kappa_N)
    $$
    as witnesses. Then
    \begin{enumerate}
        \item $2^\omega>\mathrm{cov}(\kappa_{N-1})$; and
        \item there exists a cardinal $\lambda$ of cofinality $\omega$ so that $\mathrm{pp}(\lambda)<\mathrm{cov}(\lambda)$, i.e., a cardinal at which the pseudopower and the covering number are separated.
    \end{enumerate}
If Conjecture \ref{conj:infinite} is consistent, then we have (2) and (1'), where (1') states
\begin{enumerate}
    \item[$1.'$] $2^\omega\geq\sup_n\mathrm{cov}(\kappa_n)$ 
\end{enumerate}
\end{lem}
\begin{proof}
    We prove (1) by contradiction, so suppose that $2^\omega\leq\mathrm{cov}(\kappa_{N-1})$. As 
    $$
    \kappa_{N-1}^\omega=2^\omega\cdot \mathrm{cov}(\kappa_{N-1}) \text{ and }
    2^\omega\leq\mathrm{cov}(\kappa_{N-1}),
    $$
    we conclude that $\kappa_{N-1}^\omega=\mathrm{cov}(\kappa_{N-1})$ and hence that $\mathrm{cov}(\kappa_{N-1})^\omega=\mathrm{cov}(\kappa_{N-1})$. Since $\kappa_{N-1}<\kappa_N<\mathrm{cov}(\kappa_{N-1})$, we see that
        $$
        \kappa_{N-1}^\omega\leq\kappa_N^\omega\leq\mathrm{cov}(\kappa_{N-1})^\omega=\mathrm{cov}(\kappa_{N-1})<\mathrm{cov}(\kappa_N)
        $$
        and hence that $\kappa_N^\omega<\mathrm{cov}(\kappa_N)$, contradicting the fact that $\mathrm{cov}(\kappa_N)\leq\kappa_N^\omega$.

        For $(1')$, find an $N$ so that $2^\omega\leq\mathrm{cov}(\kappa_N)$, and run the above argument with $N$ in place of $N-1$.

        For (2), we will use the inverse monotonicity of the pseudopower. We divide into cases based on where $\mathrm{pp}(\kappa_1)$ lives. If $\mathrm{pp}(\kappa_1)\leq\kappa_N$, then we've separated the pseudopower and covering numbers of $\kappa_1$, since $\kappa_N<\mathrm{cov}(\kappa_1)$. On the other hand, if $\mathrm{pp}(\kappa_1)\geq\kappa_N$, then by inverse monotonicity of the pseudopower, we know that $\mathrm{pp}(\kappa_N)\leq\mathrm{pp}(\kappa_1)\leq\mathrm{cov}(\kappa_1)<\mathrm{cov}(\kappa_N)$, and hence we've separated the pseudopower from the covering number at $\kappa_N$.
        
\end{proof}

\begin{thm}\label{th:K_many}
Fix an integer $N\ge 2$, and suppose that Conjecture \ref{conj:finite} is consistent. Work in a model with $\kappa_1,\dots,\kappa_N$ as witnesses.

Then it is consistent that $\mathfrak{b}<\kappa_1$, that the covering numbers of the $\kappa_n$ are as in Conjecture \ref{conj:finite}, and that for each $n=1,\dots,N$, there exist separable metrizable spaces, $M_n$, so that $\mathrm{kc}(M_n)=\mathrm{cov}(\kappa_1)$ while $k(M_n)=\mathrm{cov}(\kappa_n)$.

Hence, in this model,
$\sigma_k(\mathrm{cov}(\kappa_1))\geq N$.
\end{thm}

\begin{proof}
Fix $N\ge 2$. 
Start with a model in which Conjecture \ref{conj:finite} holds, and let $\kappa_1,\dots,\kappa_N$ be witnesses. We know from Lemma \ref{lem:consequences} that $2^\omega$ is larger than $\mathrm{cov}(\kappa_{N-1})$. If $\mathfrak{c}\leq\mathrm{cov}(\kappa_N)$, then we do random real forcing to increase the value of $\mathfrak{c}$ strictly above $\mathrm{cov}(\kappa_N)$, while keeping the values of $\mathfrak{b}$ and $\mathfrak{d}$ fixed (see \cite{Kunen2011}, Lemma IV.7.36). By Lemma~\ref{lem:cccPreservesCapturing}, this forcing doesn't change the relevant values of the covering numbers.

For each $1\leq n\leq N$, let $N_n$ be a space of size $\kappa_n$ which is not locally small. Next, let $N$ be any space of size $\kappa_1$ and set $M_\infty:=N^\omega$. Finally, define
$$
M_n:=M_\infty\oplus N_n
$$
for every $1\leq n\leq N$.

We now compute the compact-covering and $k$-ness numbers of each $M_n$. Beginning with the former, $kc(M_n)$ is the larger of $kc(M_\infty)$ and $kc(N_n)$. To apply Theorem~\ref{th:powers}, note that $N$ is not compact, being uncountable of size below $\mathfrak{c}$. Accordingly Theorem~\ref{th:powers} implies that 
$$
k(M_\infty)=kc(M_\infty)=\mathrm{cov}(|N|)\cdot\mathfrak{d}.
$$
Since $|N|=\kappa_1$ and since $\mathfrak{d}\leq\mathfrak{b}<\kappa_1<\mathrm{cov}(\kappa_1)$, we conclude that
$$
k(M_\infty)=kc(M_\infty)=\mathrm{cov}(\kappa_1).
$$
Now we compute $kc(N_n)$: recall that $|N_n|=\kappa_n<\mathfrak{c}$, and hence $N_n$ is totally imperfect (since it has no compact sets of size $\mathfrak{c}$, seeing as it is a space of size $<\mathfrak{c}$). Hence $kc(N_n)=|N_n|=\kappa_n\leq\kappa_N<\mathrm{cov}(\kappa_1)$. Taking the max of these two values of the compact covering number, we see that $kc(M_n)=\mathrm{cov}(\kappa_1)$.

Now we compute the $k$-ness number of $M_n$. As remarked in the previous paragraph, $k(M_\infty)=\mathrm{cov}(\kappa_1)$. By Theorem~\ref{th:k=sam}, since $N_n$ is totally imperfect but not locally small, we conclude that 
$$
k(N_n)=\mathrm{cov}(kc(N_n))\cdot\mathfrak{b}=\mathrm{cov}(\kappa_n),
$$
using our previous argument that $kc(N_n)=\kappa_n$ and the fact that $\mathfrak{b}<\kappa_1<\mathrm{cov}(\kappa_n)$. 

Now we see that $\sigma_k(\mathrm{cov}(\kappa_1))\geq N$. Indeed,
\[\{k(M):kc(M)=\mathrm{cov}(\kappa_1)\}\supseteq \{\mathrm{cov}(\kappa_n):1\leq n\leq N \} \] 
and the sequence $\la\mathrm{cov}(\kappa_n):1\leq n\leq N\ra$ is injective (recall that it is strictly increasing).
\end{proof}

\begin{thm}\label{th:ctbl_many} Suppose that Conjecture \ref{conj:infinite} is consistent, and work in a model with witnesses $\la\kappa_n:n\geq 1\ra$. Let $\lambda=\sup_{n\geq 1}\kappa_n$. Then it is consistent that $\mathfrak{b}<\kappa_1$, that the covering numbers of the $\kappa_n$ are as in Conjecture \ref{conj:infinite}, and that for each integer $n\geq 1$, there exists a separable metrizable $M_n$ so that $kc(M_n)=\mathrm{cov}(\kappa_1)$ while $k(M_n)=\mathrm{cov}(\kappa_n)$.

Hence in this model, $\sigma_k(\mathrm{cov}(\kappa_1))$ is infinite.
\end{thm}

\begin{proof}
The proof is almost identical to the proof of Theorem \ref{th:K_many}. After forcing with random real forcing, if necessary, we may assume that $\mathfrak{c}>\sup_n\mathrm{cov}(\kappa_n)$ while preserving the covering numbers. For each $n\geq 1$, let $N_n$ be a space of size $\kappa_n$ which is not locally small. Let $N$ be any space of size $\kappa_1$, and let $M_\infty:=N^\omega$. Define
$$
M_n:=M_\infty\oplus N_n,
$$
for each $n\geq 1$. As in Theorem \ref{th:K_many}, we see that the compact covering number of each $M_n$ is $\mathrm{cov}(\kappa_1)$ while the $k$-ness number of $M_n$ equals $\mathrm{cov}(\kappa_n)$.
    
Now we can see that $\sigma_k(\mathrm{cov}(\kappa_1))\geq\aleph_0$, since
$$
\{k(M):kc(M)=\mathrm{cov}(\kappa_1) \}\supseteq\{k(M_n):n\geq 1 \}=\{\mathrm{cov}(\kappa_n):n\geq 1 \}
$$
and since the sequence $\la\mathrm{cov}(\kappa_n):n\geq 1\ra$ is injective (as it is strictly increasing).
\end{proof}

There is much that we do not know about the $k$-spread and the set $\{k(M) : \mathop{kc}(M)=\kappa\}$. 
We know, from the preceding result, it is consistent  that for the $k$-spread can be at least $\aleph_0$. 
Can the $k$-spread of some uncountable $\kappa \le \mathfrak{c}$ be uncountable? 
Can $\{k(M) : \mathop{kc}(M)=\kappa\}$  be a proper subset of $[\kappa\cdot\mathfrak{b},\mathrm{cov}(\kappa)\cdot\mathfrak{b}]$? Answering these questions with the current approach seems to require building models with specific configurations of the covering numbers. By Lemma \ref{lem:consequences}, we know that this requires separating the pseudopower and covering numbers at some cardinal of cofinality $\omega$.

\paragraph{Van Douwen's Question}
In van Douwen's `Handbook' article \cite{vanDou} on small cardinals and topology he  asks the following: 
\begin{quote} Question 8.11: \ 
if $M$ is separable metrizable then is \\ \phantom{x}  $\cof(\K(M),\subseteq)=k(M)=\mathfrak{d}$ when $M$ is analytic, or at least Borel? 
\end{quote}
Evidently it is intended that $M$ not be $\sigma$-compact. 
Unasked by van Douwen, but also natural and considered in the literature, is the case when $M$ is co-analytic.

The answer is known for $\cof(\K(M),\subseteq)$. Under \textsc{CH} for any analytic $M$ we have $\cof(\K(M),\subseteq)=\mathfrak{d}=\aleph_1$. While Becker \cite{Becker} has constructed a model in which there is an analytic set $M$ with $\cof(\K(M),\subseteq)>\mathfrak{d}$. 
Van Engelen \cite{vanEng} showed that if $M$ is co-analytic but not locally compact then $\cof(\K(M),\subseteq) = \mathfrak{d}$. This also follows from earlier work of Fremlin, \cite{Fremlin}. 

It seems (see Vaughan \cite{Vau} and van Mill \cite{vanMill}) it was thought that these results completely answer  van Douwen's question. While that is true of $\cof(\K(M),\subseteq)$, it is not the case for $k(M)$. In fact our final theorem below shows that the situation with $k(M)$ is symmetrically opposite. \emph{In \textsc{ZFC}} the $k$-ness number of analytic, non $\sigma$-compact spaces is $\mathfrak{d}$. Hence for Becker's analytic $M$ we have $k(M)=\mathfrak{d}<\cof(\K(M),\subseteq)$. While it is consistent for there to be a co-analytic, non $\sigma$-compact space, $M$ say, whose $k$-ness number is strictly less than $\mathfrak{d}$. Combining this with the van Engelen/Fremlin result, for this co-analytic $M$ we have $k(M) <\mathfrak{d}=\cof(\K(M),\subseteq)$.

\begin{thm}\label{th:VDsoln} \ 
\begin{enumerate}[noitemsep,topsep=1pt]
\item[(i)] If $M$ is analytic but not $\sigma$-compact, then  $k(M)=\mathfrak{d}$. 

\item[(ii)] It is consistent there is a co-analytic, non $\sigma$-compact $M$  with $\mathop{k}(M) < \mathfrak{d}$.
\end{enumerate}
\end{thm}
\begin{proof} 
For the first part, from Theorem~3.9 of \cite{FG_Shape}, we know that for analytic non $\sigma$-compact $M$ we have 
$\mathbf{k}(M) \te \omega^\omega$, and so, indeed, $k(M)=\cof(\omega^\omega)=\mathfrak{d}$.

For part (ii) of the theorem, we work in a model where $\aleph_1^L=\aleph_1 < \mathfrak{p} \le \mathfrak{b} < \mathfrak{d}$. By Theorem 7.4 of \cite{GoldsternKellnerMejiaShelah}, this configuration is consistent: Theorem 7.4 of that paper only requires $\mathsf{GCH}$, and the poset used is a product of a c.c.c. poset and a poset which is (at least) $\omega_1$-closed. Hence, forcing with their poset over $L$ preserves $\aleph_1^L$, and in the resulting generic extension, we can obtain the desired configuration.

Next, we'd like to apply the result (simply called ``Theorem") at the beginning of Section 3.2 of \cite{MartinSolovay} to conclude that every set of reals of size $\aleph_1$ is coanalytic. To see that we can do so, note that their proof uses the separation principle they call $S_{\aleph_1}$ which (see Section 2.2 of their paper) is a consequence of MA$(\aleph_1)$ for $\sigma$-centered posets. Since $\aleph_1<\mathfrak{p}$ in our current model, Bell's theorem (\cite{Bell}) implies that MA$(\aleph_1)$ holds. Therefore, in our model, every set of reals of size $\aleph_1$ is coanalytic.

To finish the proof of (ii), take any $\aleph_1$-sized subset, $M$ say, of the reals. By the previous paragraph, $M$ is coanalytic. But $M$ is not $\sigma$-compact, in fact $\mathop{kc}(M)=\aleph_1$, because $|M|=\aleph_1<\mathfrak{c}$. Hence $\mathrm{cov}(kc(M))\cdot \mathfrak{b}=\aleph_1\cdot\mathfrak{b}=\mathfrak{b}$ (Lemma~\ref{l:sam_aleph_n}), so  $\mathop{k}(M)=\mathfrak{b}$ (Theorem~\ref{th:k_sam}) and $\mathfrak{b} <\mathfrak{d}$ in our model. 
\end{proof}

\section{Acknowledgements} Thanks to Will Brian for his interest in this work, and for informing us \cite{Brian} about the equivalence of the sampling and covering numbers. Thanks also to Todd Eisworth for helpful discussions about PCF and covering numbers.

\end{document}